\documentclass{article}
\usepackage{amsmath,amsthm,amsfonts,amssymb}
\usepackage{fullpage}
\usepackage{subfig}
\usepackage{blkarray}
\usepackage{caption}
\usepackage{float}
\usepackage{algpseudocode}
\usepackage{tikz}

\usepackage{hyperref}
\numberwithin{equation}{section}
\usepackage{xcolor}
\usepackage{colortbl}
\usepackage[normalem]{ulem}
\usepackage{sectsty}
\usepackage{mathdots}
\definecolor{astral}{RGB}{46,116,181}
\subsectionfont{\color{astral}}
\sectionfont{\color{astral}}
\usepackage{colortbl}
\usepackage{tikz}
\DeclareMathAlphabet{\mathpzc}{OT1}{pzc}{m}{it}
\DeclareFontFamily{OT1}{pzc}{}
\DeclareFontShape{OT1}{pzc}{m}{it}{<-> s * [0.900] pzcmi7t}{}
\DeclareMathAlphabet{\mathpzc}{OT1}{pzc}{m}{it}
\usepackage{longtable}
\usepackage{amsbsy}
\usepackage{amsmath}
\usepackage{accents}
\newlength{\dhatheight}

\newenvironment{frcseries}{\fontfamily{frc}\selectfont}{}

\newenvironment{key}
{\par\textbf{Keywords:\,}\hangindent=2.5cm\hangafter=1}
{\par}

\newenvironment{AMS}
{\par\textbf{Mathematics Subject Classification: }\hangindent=2.2cm\hangafter=1}
{\par}

\usepackage[a4paper,
  top=25mm,
  bottom=25mm,
  left=20mm,
  right=20mm
]{geometry}

\DeclareMathAlphabet\mathbfcal{OMS}{cmsy}{b}{n}
\definecolor{darkslategray}{rgb}{0.18, 0.31, 0.31}
\definecolor{warmblack}{rgb}{0.0, 0.26, 0.26}

\usepackage{multirow}
\usepackage{arydshln}
\usepackage{mathtools}
\usepackage{algorithm}
\usepackage{algorithmicx}
\makeatletter
\def\BState{\State\hskip-\ALG@thistlm}
\makeatother

\newtheorem{theorem}{Theorem}[section]
\newtheorem{lemma}[theorem]{Lemma}

\theoremstyle{definition}
\usepackage{hyperref}
\newtheorem{definition}{Definition}[section]
\newtheorem{remark}{Remark}[section]
\newtheorem{example}{Example}[section]

\title{Solving Coupled Tensor Equation $\mathcal{A} \ltimes \mathcal{X} =\mathcal{B}, \ \mathcal{X}\ltimes \mathcal{C}=\mathcal{D}$ using Semi-Tensor Products in the t-product}
\author{\small  Mansi Wankhede \thanks{Department of Mathematics, NIT Raipur,
Raipur-492010, India({\tt mansiwankhede0714@gmail.com}). }, Ranjan Kumar Das \thanks{Department of Mathematics, NIT~Raipur,
Raipur-492010, India,({\tt rkdas.maths@nitrr.ac.in}, ). } 
\date{}}
\begin{document}
\maketitle
\begin{abstract}
This paper investigates the solution of coupled third-order tensor equation $\mathcal{A} \ltimes \mathcal{X} = \mathcal{B},\ \mathcal{X} \ltimes \mathcal{C} = \mathcal{D},$ of arbitrary dimensions by incorporating semi-tensor product (STP) within t-product framework, where the unknown $\mathcal{X}$ can take  form of  vector, matrix, or tensor. For the unknown $\mathcal{X}$, we establish a necessary and sufficient condition that provides an equivalence criterion for the existence of solutions. Moreover, the explicit structure (Toeplitz, Circulant) of $\mathcal{C}$ and $\mathcal{D}$ is characterized. Theoretical results are supported by several illustrative examples.
\end{abstract}

\begin{key}
    Coupled Tensor Equation, Semi Tensor Product, t-product, Compatibility Conditions, Toeplitz and Circulant Tensor.
\end{key}

\begin{AMS}
15A69, 15A60.
\end{AMS}

\section{Introduction}
Tensors have emerged as an essential mathematical tool for representing and analyzing multiway data in modern science and engineering \cite{ A.Be2025, A.BE2024,J.Coopertensorapplication,T.Kolda2009Tensordecomposition,D.mishra2017,D.Mishra2022,D.mishra2025}. Unlike vectors and matrices, tensors naturally capture higher-order interactions and multidimensional structures, making them particularly suitable. This paper addresses the solution of coupled third-order tensor equations of arbitrary dimensions by incorporating semi-tensor product (STP) within t-product framework.or applications such as signal and image processing, video analysis, data mining, control systems, networked systems, and machine learning \cite{D.S.~Burdicktensorapplication,H.Jintensorapplication,Y.Jitensorapplication,D.mishra2020}.
 \par The algebraic manipulation of tensors depends critically on the choice of tensor products. Over the past decades, a various tensor products have been proposed. Among them the Kronecker product \cite{K.Batselier2017Kroneckerproductoftensor}, k-mode product \cite{L.Qi2017Spectraltheory}, Tucker product \cite{T.Kolda2009Tensordecomposition}, and t-product(t-P) \cite{M.Kilmer2011factorizationfortensor}. Each product induces a different algebraic structure and is suitable for specific applications. Among them, the t-P has attracted significant attention due to its ability to extend many matrix concepts, such as matrix multiplication, transpose, inverse, and singular value decomposition, to third-order tensors. Nevertheless, most tensor products, including the standard t-product, impose strict dimensional compatibility requirements, which substantially restrict their applicability in practical problems involving heterogeneous tensor dimensions.
\par To address dimensional incompatibility, Cheng et al. \cite{D.cheng2012STPofmatrices} introduced the semi-tensor product(STP) as an extension of the conventional matrix multiplication. The STP removes the requirement of equal inner dimensions by embedding matrices into higher-dimensional spaces via the Kronecker product. As a result, the STP coincides with the standard matrix product when dimensions are compatible, while remaining well defined for arbitrary dimensions. Due to its flexibility and its algebraic properties, the STP has become an important category of generalized matrix operations and has been successfully applied to the analysis of logical networks, Boolean control systems, evolutionary games, cryptographic algorithms, and networked dynamical systems \cite{W.Liu2022STPoftensor}.
\par Matrix equations of the form
\textit{$AX=B$}
play a fundamental role in numerical analysis, control theory, and systems see\cite{J.Yao2016solutionAX=B,J.2017STPapplication} and the references there in. Within the STP framework, this classical equation has been extended to cases with incompatible dimensions and further generalized to tensor equations. By employing the STP, its solvability conditions and solution methods for both matrix and tensor versions of the equation $A \ltimes X = B$, where the unknown variable may be a vector, matrix, or tensor \cite{J.FathitensorequationAX=BunderSTP}. These developments significantly broaden the applicability of linear equation theory to multilinear and high-dimensional systems \cite{J.FathitensorequationAX=BunderSTP}.
Beyond single-sided equations, coupled matrix equations of the form
$AX=B, \ XC = D$
have also been studied under the STP framework. In particular, Li et al. \cite{J.Li2017AX=BCX=DunderSTP} investigated the solvability of such coupled equations for matrices and derived necessary and sufficient conditions together with explicit solution methods. These results demonstrate that the STP provides a powerful tool for analyzing constrained linear systems involving multiple interactions. However, the in \cite{J.Li2017AX=BCX=DunderSTP} studies are primarily confined to the matrix setting and do not fully exploit the potential of tensor representations.

 Motivated by the work  
 \cite{J.Li2017AX=BCX=DunderSTP}, the present work aims to extend the coupled equation $AX = B, \, XC = D$ from matrices to third-order tensors by incorporating the STP within the t-P framework. Specifically, we investigate the solvability of the coupled tensor equation
\begin{equation}\tag{1}\label{1}
\mathcal{A} \ltimes \mathcal{X} = \mathcal{B}, \ \mathcal{X} \ltimes \mathcal{C} = \mathcal{D},
\end{equation}
where $\mathcal{A},\ \mathcal{B},\ \mathcal{C},\ \mathcal{D}$ are given third-order tensors and $\mathcal{X}$ is an unknown vector, matrix or tensor. By combining the dimensional flexibility of the STP with the algebraic structure of the t-P, we develop a unified and systematic framework for analyzing solvability conditions and constructing solutions for coupled tensor equation with arbitrary dimensions. The coupled tensor equation \eqref{1} studied in this paper have important applications in several areas. In networked noncooperative games, such equations naturally arise in the characterization of equilibrium conditions and strategy updates involving multiway interactions among agents.
\par The structure of the paper is as follows. Section~2 presents the necessary preliminaries on the t-product and the STP for third-order tensors. Section~3 investigates the solvability of the coupled tensor equation $\mathcal{A} \ltimes \mathcal{X} = \mathcal{B}$ and $\mathcal{X} \ltimes \mathcal{C} = \mathcal{D}$, deriving necessary and sufficient conditions together with explicit solution methods. Based on this framework, we analyze the capability of the tensor–matrix equation and tensor-tensor equation using STP under t-product in Section 4 and Section 5. Finally, concluding remarks are presented in the last section.
\section{Preliminaries}
In this section, we introduce the notation and basic definitions required in the subsequent analysis. $\mathbb{R}$(resp.$\mathbb{C}$) denote the sets of real numbers(resp. complex numbers), and $\mathbb{R}^n$(resp.$\mathbb{C}^n$) consist of real column vectors (resp. complex column vectors) of dimension $n$, while $\mathbb{R}^{m \times n}$  (resp.$\mathbb{C}^{m \times n}$) represent the sets of real matrices(resp. complex matrices) of size $m \times n$. Matrices are denoted by uppercase letters, vectors by italicized uppercase letters, and scalars by lowercase letters. For a matrix $A$, the entry in the $i^{th}$ row and $j^{th}$ column is denoted by $a_{ij}$, and its $i^{th}$ column is written as $A_i$. For positive integers $m$ and $n$, their least common multiple and greatest common divisor are denoted by $[m,n]$ and $(m,n)$, respectively.\\
A tensor is a multidimensional array whose order is defined as the number of indices required to specify one of its elements. From this perspective, scalars, vectors, and matrices can be referred as tensors of order zero, one, and two, respectively. An $N$-th order tensor $\mathcal{A} \in \mathbb{R}^{n_1 \times n_2 \times \cdots \times n_N}$ consists of entries denoted by $a_{i_1 i_2 \cdots i_N}$, where $1 \leq i_j \leq n_j$ for $j = 1, \ldots, N$. Each element corresponds uniquely to the index tuple $(i_1, i_2, \ldots, i_N)$.\\
For a tensor $\mathcal{A} \in \mathbb{R}^{n_1 \times n_2 \times \cdots \times n_N}$, fixing the last index yields a subtensor obtained by slicing $\mathcal{A}$ along its final mode. $\mathcal{A}_{(: \, : \, \cdots \, : \, k)}, \ k = 1,2,\ldots,n_N,$
which lies in $\mathbb{R}^{n_1 \times n_2 \times \cdots \times n_{(N-1)}}$ and is referred to as a \emph{frontal slice}. More generally, \emph{fibers} are obtained by fixing all but except one index of the tensor, serving as the higher-order generalization of matrix rows and columns.\cite{J.FathitensorequationAX=BunderSTP}
\begin{figure}[htbp]
    \centering
    \includegraphics[width=0.4\textwidth]{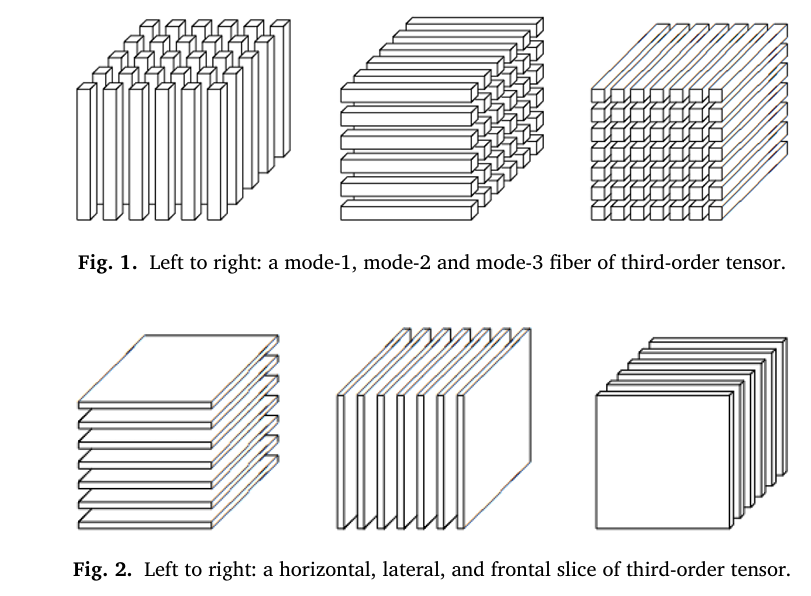}
    \label{fig:example}
\end{figure}\\
In this paper, we restrict our attention to third-order tensors, which are denoted by calligraphic letters. A typical third-order tensor is written as
$\mathcal{A} = (a_{ijk}),$
where $1 \leq i \leq n_1$, $1 \leq j \leq n_2$, and $1 \leq k \leq n_3$. The notations $\mathcal{A}_{(i:k)}$, $\mathcal{A}_{(:jk)}$, and $\mathcal{A}_{(ij:)}$ represent fibers along the first, second, and third modes, respectively (see Fig.~1). Similarly, $\mathcal{A}_{(i::)}$, $\mathcal{A}_{(::k)}$ and $\mathcal{A}_{(:j:)}$ denote the horizontal, frontal and lateral slices of $\mathcal{A}$, corresponding to fixed indices in the first, second, and third modes, respectively (see Fig.~2). Each slice is a matrix. Based on this representation, a tensor multiplication framework is developed by interpreting a third-order tensor as an ordered collection of its frontal slices. For a tensor $\mathcal{A} \in \mathbb{R}^{m \times n \times r}$, it consists of $r$ frontal slices of size $m \times n$, which are denoted by
$
\mathcal{A}_{(::1)}, \mathcal{A}_{(::2)}, \ldots, \mathcal{A}_{(::r)}.
$
\begin{definition}(Identity Tensor
\cite{Z.Chen2023SVDlikedecomposition}) An $n \times n \times k$ tensor $\mathcal{I}{n\times n\times k}$ is referred to as an identity tensor if its first frontal slice equals to $n \times n$ identity matrix, while all remaining frontal slices consist entirely of zeros. In the particular case $k=1$, the tensor $\mathcal{I}{n\times n\times 1}$ reduces to an $n \times n \times 1$ identity tensor whose single frontal slice is the $n \times n$ identity matrix, for convenience, this tensor is denoted by $\mathcal{I}_{n\times n}$.
\end{definition}
\begin{definition}(t-Product
\cite{M.Kilmer2008SVDtensor,M.Kilmer2011factorizationfortensor})
Let $\mathcal{A}\in\mathbb{C}^{m\times n\times r}$ and $\mathcal{B}\in\mathbb{C}^{n\times s\times r}$ be third-order tensors. Their t-product(t-P), denoted by $\mathcal{A} * \mathcal{B}$, yields a tensor of size $m\times s\times r$ and is defined by
$$
\mathcal{A} * \mathcal{B}
=
\mathrm{fold}\!\left(
\mathrm{bcirc}(\mathcal{A}).\,
\mathrm{unfold}(\mathcal{B})
\right).
$$
Here, the operator $\mathrm{unfold}(\cdot)$ stacks the frontal slices of a tensor vertically and  $\mathrm{bcirc}(\mathcal{A})$ denotes the block circulant matrix constructed from the frontal slices of $\mathcal{A}$, given by $$
\mathrm{unfold}(\mathcal{A}) =
\begin{bmatrix}
\mathcal{A}_{(::1)} \\
\mathcal{A}_{(::2)} \\
\vdots \\
\mathcal{A}_{(::r)}
\end{bmatrix} 
\in \mathbb{C}^{mr \times n},\qquad
\mathrm{bcirc}(\mathcal{A}) =
\begin{bmatrix}
\mathcal{A}_{(::1)}   & \mathcal{A}_{(::r)}   & \cdots & \mathcal{A}_{(::2)} \\
\mathcal{A}_{(::2)}   & \mathcal{A}_{(::1)}   & \cdots & \mathcal{A}_{(::3)} \\
\vdots              & \vdots              & \ddots & \vdots            \\
\mathcal{A}_{(::r)}   & \mathcal{A}_{(::r-1)} & \cdots & \mathcal{A}_{(::1)}
\end{bmatrix}.$$
The operator $\operatorname{fold}$ is defined as the inverse of $\operatorname{unfold}$ mapping. The symbol $`*`$ is used to denote the product operation between third-order tensors, referred to as the $t$-P, while $`\cdot`$ represents the standard matrix product. The $t$-P plays a fundamental role in third-order tensor analysis, as it extends classical matrix multiplication to the tensor setting and introduces an algebraic structure with properties analogous to those of matrix operations.
\end{definition}
\begin{definition}(Kronecker Product
\cite{K.Batselier2017Kroneckerproductoftensor,W.Liu2022STPoftensor})Let $\mathcal{A} = (a_{i_1 i_2 \cdots i_m}) \in \mathbb{R}^{n_1 \times n_2 \times \cdots \times n_m}$
 \text{and}
$\mathcal{B} = (b_{j_1 j_2 \cdots j_m}) \in \mathbb{R}^{p_1 \times p_2 \times \cdots \times p_m},$ be two tensors of order $m$. The Kronecker product of $\mathcal{A}$ and $\mathcal{B}$ denoted by $\mathcal{A}\otimes\mathcal{B}$, is an $m$-order tensor of size
$n_1p_1 \times n_2p_2 \times \cdots \times n_mp_m$, defined by scaling the tensor $\mathcal{B}$ with each entry of $\mathcal{A}$, that is,
\[
\mathcal{A} \otimes \mathcal{B}
:= (a_{i_1 i_2 \cdots i_m}\mathcal{B})
\in \mathbb{R}^{n_1 p_1 \times n_2 p_2 \times \cdots \times n_m p_m},
\]
Specifically, each element of $\mathcal{A}\otimes\mathcal{B}$ is given by
\[
(\mathcal{A} \otimes \mathcal{B})_{[i_1 j_1]\,[i_2 j_2]\cdots[i_m j_m]}
=
a_{i_1 i_2 \cdots i_m}\,
b_{j_1 j_2 \cdots j_m}.
\]
where the combined index $[i_kj_k]$ corresponds to the single index $(i_k-1)p_k + j_k$, with $p_k$ denoting the size associated with the index $j_k$ for $k=1,2,\ldots,m$.
\end{definition}
\begin{definition}(Semi-Tensor Product(STP) \cite{J.FathitensorequationAX=BunderSTP})Let $\mathcal{A}=(a_{i_1 i_2 i_3}) \in \mathbb{C}^{m\times n\times r}$ and
$\mathcal{B}=(b_{j_1 j_2 j_3}) \in \mathbb{C}^{p\times q\times s}$ be two
third-order tensors. The STP of
$\mathcal{A}$ and $\mathcal{B}$, denoted by $\mathcal{A}\ltimes\mathcal{B}$,
is defined as
\[
\mathcal{A}\ltimes\mathcal{B}
=
\left(
\mathcal{A}\otimes I_{\frac{t}{n}\times\frac{t}{n}\times\frac{d}{r}}
\right)
*
\left(
\mathcal{B}\otimes I_{\frac{t}{p}\times\frac{t}{p}\times\frac{d}{s}}
\right)
\in \mathbb{C}^{\frac{mt}{n}\times\frac{qt}{p}\times d},
\]
where $t=[n,p]$ and $d=[r,s]$.
In the particular case $n=p$ and $r=s$, the above definition reduces to the standard $t$-P. Moreover, the STP preserves the main algebraic properties of the $t$-P. Compared with the $t$-product, the STP allows the multiplication of tensors whose dimensions are not compatible. Finally, when restricted to second-order tensors, this definition coincides with the classical STP of matrices.
\end{definition}
\section{Solvability conditions for tensor equations with vector Solutions}
This section is devoted in solving the tensor-vector equation for unknown vector $X$ given as:
\begin{equation}\label{eq:3.1}
\left\{
\begin{aligned}
& \underbrace{\mathcal{A} \ltimes X}_{\frac{m t_1}{n} \times \frac{t_1}{p} \times d_1} = \mathcal{B} \\
& \underbrace{X \ltimes \mathcal{C}}_{pt_{2} \times \frac{bt{2}}{a}\times d_2} = \mathcal{D},
\end{aligned}
\right.
\end{equation}\\
here, $\mathcal{A}\in \mathbb{C}^{m \times n \times r}$, $\mathcal{B}\in \mathbb{C}^{h \times k \times r}$, $\mathcal{C}\in \mathbb{C}^{a \times b \times r}$, and $\mathcal{D}\in \mathbb{C}^{l \times d \times r}$ are given tensors, while $X \in \mathbb{C}^{p}$ denotes the unknown vector. Define $t_{1} =[n,p]$, $t_{2} = [1, a]$, and $d_{1} = d_{2} = [1, r]$. The analysis is initiated for the special configuration $m=h$ with $l=p$, after which the results are extended to the more general case.
\subsection{The simplified case with \texorpdfstring{$m=h$}{m=h}}
In this section, we investigate the existence of the solution  of the tensor–vector Eq\eqref{eq:3.1}. The two simplified cases, $m=h$ with $l=p$ and $m=h$ with $l\neq p$ lead to similar solvability conditions for Eq\eqref{eq:3.1}. Therefore, it suffices to focus on the case $m=h$. The corresponding results can be readily derived using the STP.
\begin{lemma}\label{lemma:3.1}
Assume that Eq\eqref{eq:3.1} admits a solution $X$. Then, the dimensions of the tensors $\mathcal{A}, \mathcal{B}, \mathcal{C}$, and $\mathcal{D}$ necessarily satisfies the following conditions:\vspace{-0.5em}
\begin{enumerate}
    \item[(i)] The ratios $\frac{n}{k}$ and $\frac{l}{a}$ are positive integers; (ii) $b = d$ and $p = \frac{l}{a} = \frac{n}{k}$.
\end{enumerate}
\end{lemma}
\begin{proof}Let $X\in \mathbb{C}^p$ satisfy Eq\eqref{eq:3.1},
\begin{equation}
    \underbrace{\mathcal{A} \ltimes X}_{\frac{m t_1}{n} \times \frac{t_1}{p} \times d_1} = \mathcal{B},
\end{equation}\label{eq:3.2}
$$\frac{mt_{1}}{n}=h, \frac{t_{1}}{p}=k,$$\vspace{-0.5em}
here $ t_{1}=[n, p] \text{ and }d_{1}=[r,1]=r .$
Thus $t_{1}=n$, and $p=\frac{n}{k}$ is a positive integer. Then
\begin{equation}\label{eq:3.3}
\underbrace{X \ltimes \mathcal{C}}_{{p t_2} \times \frac{bt_2}{a} \times d_2} = \mathcal{D},
\end{equation}
$$
pt_{2}=l, \frac{bt_{2}}{a}=d,$$ 
$ \text{ where } t_{2}=[1, a] \ \text{ and } \ d_{2}=[r,1]=r.$\\
Hence, we have $t_{2} = n$ and  $\frac{l}{a}$ is a positive integer, $p = \frac{l}{a} = \frac{n}{k}$, and $b = d$. Consequently, Eq\eqref{eq:3.1} admits a solution only when these two conditions are satisfied, which concludes the proof.
\end{proof}
At this section, we investigate the sufficient and necessary conditions for the solvability of Eq\eqref{eq:3.1}. Decompose tensor $\mathcal{A}\text{ as }[\widetilde{\mathcal{A}}_{1}\ \widetilde{\mathcal{A}}_{2}\ \cdots \ \widetilde{\mathcal{A}}_{p}],$ where each sub-tensor $\mathcal{\widetilde{A}}_{g}\in\mathbb{C}^{h \times k\times r}$  and using the STP, the equation $\mathcal{A} \ltimes X = \mathcal{B}$ can be equivalently reformulated as\vspace{-0.7em}
\begin{equation}
\mathcal{A}\ltimes X=[\widetilde{\mathcal{A}}_{1}\ \widetilde{\mathcal{A}}_{2}\ \cdots \ \widetilde{\mathcal{A}}_{p}]\ltimes
\left[\begin{array}{cccc}
    x_{1}& x_{2}&\cdots&x_{p}
\end{array}\right]^T=\sum_{g=1}^px_{g}\widetilde{\mathcal{A}}_{g}=\mathcal{B}\in \mathbb{C}^{h \times k \times r},
\end{equation}\vspace{-0.7em}
and Eq\eqref{eq:3.3} could be described as 
\begin{equation}\label{eq:3.5}
X\ltimes \mathcal{C}=(X\otimes I_{a\times a\times r})*\mathcal{C}=
\begin{bmatrix}
    x_{1}\mathcal{C}&x_{2}\mathcal{C}&\cdots&x_{n}\mathcal{C}
      \end{bmatrix}^T=\begin{bmatrix}
          \mathcal{D}_{1}&\mathcal{D}_{2}&\cdots& \mathcal{D}_{p}\\
      \end{bmatrix}^T=  \mathcal{D},
      \end{equation}
      Here, sub-tensor of $\mathcal{D}$ is denoted by $\mathcal{D}_{g} \in \mathbb{C}^{a \times b \times r}$. By following the method outlined above, the subsequent results can be directly obtained.\\
\vspace{-1em}    
\begin{lemma}\label{lemma:3.2}
Eq\eqref{eq:3.1} admits a solution if and only if the following two conditions hold:\vspace{-0.4em}
\begin{enumerate}
    \item[(i)] The tensors $\widetilde{A}_1, \widetilde{A}_2, \ldots, \widetilde{A}_p$ and $\mathcal{B}$ are linearly dependent;\vspace{-0.4em}
    \item[(ii)] The entries of $\mathcal{D}$ satisfy $d_{\,t + (g-1)a,\, j,\, i} = \alpha_g \, c_{t,j,i},$ where $c_{t,j,i} \in \mathcal{C}$, $\alpha_g \in \mathbb{C}$, $g = 1, \ldots, p$,\ $j = 1, \ldots, d$ and $i=1,\ldots,r$.
\end{enumerate}
\end{lemma}
Here, the solution vector $X$ is given by $X = [\alpha_1, \alpha_2, \ldots, \alpha_p]^T$. Moreover, the solution is unique except in the case where $\mathcal{C}$ and $\mathcal{D}$ are zero and the tensors $\widetilde{A}_1, \widetilde{A}_2, \ldots, \widetilde{A}_p$ remain linearly dependent according to condition (i).
 \begin{definition}(Column Stacking \cite{D.cheng2012STPofmatrices})
Let $A=(a_{ij})\in \mathbb{C}^{m \times n}.$ The column stacking form of $A,$ denoted by $V_c(A),$ is defined as 
$$ V_c(A):= (a_{11},\ldots,a_{m1},a_{12},\ldots,a_{m2},\ldots,a_{1n},\ldots,a_{mn})^T.$$
\end{definition}
\begin{definition}(Lateral form)
Let $\mathcal{A}\in \mathbb{C}^{n_1\times n_2 \times n_3}.$ Then lateral form of $\mathcal{A},$ denoted by $\mathcal{V}_L(\mathcal{A}),$ is defined as
$$\mathcal{V}_L(\mathcal{A}):=[V_c(\mathcal{A}(:1:)), V_c(\mathcal{A}(:2:)), \ldots,V_c(\mathcal{A}(:n_2:)) ]^T.$$
\end{definition}
\begin{example}
    Let us consider $\mathcal{A}$ as follows:
$$\begin{tabular}{cc|cc}
 \hline
\multicolumn{2}{c}{${{\mathcal{A}}}(:,:,1)$} & 
\multicolumn{2}{c}{${\mathcal{A}}(:,:,2)$}  \\
\hline
 2&0&1&3 \\
 6&1&5&2\\
 \end{tabular}, $$ then the lateral form of $\mathcal{A}$ is $\mathcal{V}_L(\mathcal{A})=[2,6,1,5,0,1,3,2]^T.$ 
\end{example}
By  using Lemma~\ref{lemma:3.2}, we derive  the following result.
\begin{theorem}\label{th3.3}
 Eq\eqref{eq:3.1} is solvable if and only if the following tensor–vector equation, expressed in terms of the classical matrix product is solvable:
\begin{equation}
\left\{
\begin{aligned}
&x_1 \mathcal{V}_L(\widetilde{\mathcal{A}}_1) + x_2 \mathcal{V}_L(\widetilde{\mathcal{A}}_2) + \cdots + x_p \mathcal{V}_L(\widetilde{\mathcal{A}}_p)
= \mathcal{V}_L(\mathcal{B}) \in \mathbb{C}^{rhk}, \\
&x_g \mathcal{V}_L(\mathcal{C}) = \mathcal{V}_L(\mathcal{D}_g) \in \mathbb{C}^{rab },
\end{aligned}
\right.
\end{equation}\\
where $
\mathcal{V}_L(\widetilde{\mathcal{A}}_g) =
\begin{bmatrix}
\mathcal{A}_1 &\mathcal{A}_2 &\cdots &\mathcal{A}_k
\end{bmatrix}^T,
\quad
\mathcal{V}_L(\mathcal{D}_g) =
\begin{bmatrix}
\mathcal{D}_{g1}&\mathcal{D}_{g2}&\dots&\mathcal{D}_{gd}
\end{bmatrix}^T
$,\
 $\mathcal{A}_{g}$ denote the $g^{th}$ lateral slice of $\mathcal{A}$, and let $\mathcal{D}_{g} \in \mathbb{C}^{a \times b \times r}$ be a sub-tensor of $\mathcal{D}$. Further, $\mathcal{D}_{g j}$ represents the $j^{th}$ lateral slice of the sub-tensor $\mathcal{D}_{g}$, where $g = 1, \ldots, p$ and $j = 1, \ldots, d$.
\end{theorem}
Next, we provide example illustrating the above Lemma \ref{lemma:3.1}, Lemma \ref{lemma:3.2} and Theorem \ref{th3.3}.
\begin{example}
 Let $\mathcal{A},\mathcal{B},\mathcal{C}$ and $\mathcal{D}$ be $3^{rd}$ order tensors defined as
 \begin{center}
\begin{tabular}{c c c| c c c}
\hline
  \multicolumn{3}{c}{$\mathcal{A}(:,:,1)$} & 
  \multicolumn{3}{c}{$\mathcal{A}(:,:,2)$} \\
  
\hline
1&2&1&2&1&0 \\
1&0&1&3&1&0 \\
0&1&1&0&0&1\\
\hline
\end{tabular},\quad
   \begin{tabular}{c|c}
\hline
  \multicolumn{1}{c}{$\mathcal{B}(:,:,1)$} & 
  \multicolumn{1}{c}{$\mathcal{B}(:,:,2)$} \\ 
  \hline
8&4 \\
4&5 \\
5&3\\
\hline
\end{tabular},\quad \end{center}
\begin{center}
 \begin{tabular}{c c | c c}
\hline
  \multicolumn{2}{c}{$\mathcal{C}(:,:,1)$} & 
  \multicolumn{2}{c}{$\mathcal{C}(:,:,2)$} \\
  
\hline
1&1&1&0 \\
2&1&2&2 \\
\hline
\end{tabular},\quad and 
 \begin{tabular}{c c| c c}
\hline
  \multicolumn{2}{c}{$\mathcal{D}(:,:,1)$} & 
  \multicolumn{2}{c}{$\mathcal{D}(:,:,2)$} \\
  \hline
1&1&1&0 \\
2&1&2&2 \\
2&2&2&0\\
4&2&4&4\\
3&3&3&0\\
6&3&6&6\\
\hline
\end{tabular}.
\end{center}
Then, the solution $X\in \mathbb{C}^3$, and $\widetilde{\mathcal{A}}_{1}\ \widetilde{\mathcal{A}}_{2}$ and $\widetilde{\mathcal{A}}_{3}$ with $\mathcal{B}$ are linearly dependent, where 
\begin{center}
 $\widetilde{\mathcal{A}}_{1}=$\begin{tabular}{c|c}
\hline 
\multicolumn{1}{c}{ $\widetilde{\mathcal{A}}_{1}(:,:,1)$} & 
\multicolumn{1}{c}{$\widetilde{\mathcal{A}}_{1}(:,:,2)$} \\
   \hline 
   1&2\\
   1&3\\
   0&0\\
   \hline
 \end{tabular},
 $\widetilde{\mathcal{A}}_{2}=$ \begin{tabular}{c|c}
\hline
  \multicolumn{1}{c}{$\widetilde{\mathcal{A}}_{1}(:,:,1)$} & 
  \multicolumn{1}{c}{$\widetilde{\mathcal{A}}_{1}(:,:,2)$} \\
   \hline 
   2&1\\
   0&1\\
   1&0\\
   \hline
\end{tabular}, and
 $\widetilde{\mathcal{A}}_{3}=$ \begin{tabular}{c|c}
\hline
  \multicolumn{1}{c}{$\widetilde{\mathcal{A}}_{1}(:,:,1)$} & 
  \multicolumn{1}{c}{$\widetilde{\mathcal{A}}_{1}(:,:,2)$} \\
   \hline 
   1&0\\
   1&0\\
   1&1\\
   \hline
   \end{tabular}.\\
\end{center}
\begin{center}
      $\widetilde{\mathcal{D}}_{1}=$\begin{tabular}{c c|c c}
\hline
  \multicolumn{2}{c}{$\widetilde{\mathcal{D}}_{1}(:,:,1)$} & 
  \multicolumn{2}{c}{$\widetilde{\mathcal{D}}_{1}(:,:,2)$} \\
   \hline 
   1&1&1&0\\
   2&1&2&2\\
   \hline
   \end{tabular},
  $\widetilde{\mathcal{D}}_{2}=$\begin{tabular}{c c|c c}
\hline
  \multicolumn{2}{c}{$\widetilde{\mathcal{D}}_{2}(:,:,1)$} & 
  \multicolumn{2}{c}{$\widetilde{\mathcal{D}}_{2}(:,:,2)$} \\
   \hline 
   2&2&2&0\\
   4&2&4&4\\
   \hline
   \end{tabular}, and 
     $\widetilde{\mathcal{D}}_{3}=$\begin{tabular}{c c|c c}
\hline
  \multicolumn{2}{c}{$\widetilde{\mathcal{D}}_{3}(:,:,1)$} & 
  \multicolumn{2}{c}{$\widetilde{\mathcal{D}}_{3}(:,:,2)$} \\
   \hline 
   3&3&3&0\\
   6&3&6&6\\
   \hline
\end{tabular},
\end{center}
together with $\mathcal{C}$, these tensors are also linearly dependent. It can be easily verified that the conditions of Lemma~\ref{lemma:3.2} are satisfied. Therefore, the equation admits a unique solution given by $X = [1, 2, 3]^T$.
\end{example}
\subsection{The gereral case corresponds to \texorpdfstring{$m\neq h$}{m=h }} 
We investigate the existence of the solution of tensor–vector Eq\eqref{eq:3.1} under two cases, namely 
$m\neq h$ with $l=p$ and
$m\neq h$ with $l\neq p$. Since the analysis of these two cases are essentially analogous, we restrict our attention to the case $m\neq h$. By an argument analogous to that of Lemma~\ref{lemma:3.1}, we obtain the following result.
 \begin{lemma}\label{lemma:3.4}
Assume that Eq\eqref{eq:3.1} admits a solution $X$. Then, the dimensions of the tensors $\mathcal{A}, \mathcal{B}, \mathcal{C}$, and $\mathcal{D}$ necessarily satisfies the following:\vspace{-0.4em}
\begin{enumerate}
    \item[(i)] The ratios $\frac{h}{m}$ and $\frac{n}{k}$ are positive integers, with $(k,\frac{h}{m}) = 1$; (ii) The ratio $\frac{l}{a}$ is a positive integer and $\frac{l}{a} \neq \frac{n}{k}$. Moreover, $p = \frac{l}{a} = \frac{nh}{mk}$ and $b = d$.
\end{enumerate}
\end{lemma}

 \begin{proof}
 Let $X \in \mathbb{C}^p$ satisfy Eq\eqref{eq:3.1}. Then, from Eq\eqref{eq:3.1}, we obtain
$$\frac{mt_{1}}{n}=h, \frac{t_{1}}{p}=k, \qquad pt_{2}=l, \frac{bt_{2}}{a}=d,$$
 where $t_{1}=[n, p] \text{ and }d_{1}=[r,1]=r , \ t_{2}=[1, a] \ \text{ and } \ d_{2}=[r,1]=r$. Thus $\frac{t_{1}}{n}=\frac{h}{m}$ is a positive integer. Then,  $t_{2}=a$ and $p=\frac{l}{t_{2}}=\frac{l}{a}=\frac{t_{1}}{k}=\frac{nh}{mk}$. Thus, $\frac{l}{a}$ is positive integer. \\
Moreover,\vspace{-0.7em}
$$ t_{1}=\frac{nh}{m}=[n,p]=\left[n,\frac{nh}{mk}\right], 
$$
\vspace{-0.5em}
Since $\frac{n}{k}$ divides $\frac{nh}{k}$, it follows that $\frac{n}{k}$ is a positive integer. Moreover, we have
$$
\frac{t_1}{n} = \frac{h}{m}, \quad \frac{t_1}{\frac{nh}{mk}} = k.$$
\vspace{-0.3em}
It then follows that $\frac{h}{m}$ and $k$ are relatively prime. Now, if $\frac{l}{a} = \frac{n}{k}$ and $p = \frac{h}{m} \cdot \frac{n}{k} = \frac{l}{a}$, we would obtain $\frac{h}{m} = 1$, which contradicts the assumption $m \neq h$. This completes the proof.
 \end{proof} \vspace{-0.5em}
 Next, we provide example illustrating the above Lemma \ref{lemma:3.4}.\vspace{-0.7em}
 \begin{example}
 Let $\mathcal{A,B,C}$ and $\mathcal{D}$ be $3^{rd}$ order tensors defined as
\begin{center}
\begin{tabular}{c c c|c c c}
\hline
\multicolumn{3}{c}{$\mathcal{A}(:,:,1)$} & 
\multicolumn{3}{c}{$\mathcal{A}(:,:,2)$} \\
\hline 
1&0&-1&0&1&1\\
\hline
\end{tabular},\quad
\begin{tabular}{c|c}
\hline
\multicolumn{1}{c}{$\mathcal{B}(:,:,1)$} & 
\multicolumn{1}{c}{$\mathcal{B}(:,:,2)$} \\
\hline 
2&1\\
-2&1\\
\hline
\end{tabular},\quad
\begin{tabular}{c c|c c}
\hline
\multicolumn{2}{c}{$\mathcal{C}(:,:,1)$} & 
\multicolumn{2}{c}{$\mathcal{C}(:,:,2)$} \\
\hline 
2&1&0&2\\
\hline
\end{tabular}
\quad
and \begin{tabular}{c c|c c}
\hline
\multicolumn{2}{c}{$\mathcal{D}(:,:,1)$} & 
\multicolumn{2}{c}{$\mathcal{D}(:,:,2)$} \\
\hline 
2&1&0&2\\
-2&-1&0&-2\\
4&2&0&4\\
0&0&0&0\\
-2&-1&0&-2\\
2&1&0&2\\
\hline
\end{tabular} 
\end{center}
By Lemma \ref{lemma:3.4}, it is straight forward to verify that the given tensors are compatible. A straight forward calculation will give the solution is $X=[1 ,-1,2,0,-1,1]^{T}$.
\end{example} \vspace{-0.6em}
We now characterize the explicit structure of $\mathcal{B}$ by exploiting the properties of the STP and the associated dimensional compatibility conditions.\vspace{-0.4em}
\begin{definition}\label{def}(Toeplitz tensor
\cite{J.FathitensorequationAX=BunderSTP}) A third-order tensor $\mathcal{A} = (a_{i_1,i_2,i_3}) \in \mathbb{C}^{n_1 \times n_2 \times n_3}$ is said to be a Toeplitz tensor if its entries depend only on the difference of the first two indices, that is,\vspace{-0.7em}
$$
a_{i_1,i_2,i_3} = g_{\,i_1 - i_2,\, i_3},
\qquad
1 \le i_1 \le n_1,\;
1 \le i_2 \le n_2,\;
1 \le i_3 \le n_3,
$$
where $\mathcal{G} = (g_{k_1,k_2}) \in \mathbb{C}^{(n_1+n_2-1) \times n_3} \ \text{with}
\ 1-n_2 \le k_1 \le n_1-1  \ \text{and}\ 1 \le k_2 \le n_3.$ In this case, $\mathcal{A}$ is generated by $\mathcal{G}$ and is denoted by $\mathcal{A} := \mathrm{toep}(\mathcal{G})$. Equivalently, a third-order tensor is a Toeplitz tensor if every frontal slice of the tensor is a Toeplitz matrix.
\end{definition}
Under the STP, the resulting tensor exhibits a block Toeplitz structure, which follows directly from the STP of two tensors.\\
We now characterize the explicit structure of $\mathcal{B}$ by exploiting the properties of the semi-tensor product and the associated dimensional compatibility conditions.
\begin{theorem}\label{theorem:3.5}
If there exists a solution to Eq\eqref{eq:3.1}, then $\mathcal{B}$ must possess a sub-tensor Toeplitz form. Moreover, the elements of
$\mathcal{D}_g$ satisfy $d_{t+(g-1)a,\,j,\,i}=\alpha_g\,c_{t,\,j,\,i}$,
where $c_{t,j,i},\ \alpha_g \in \mathbb{C}$. Each
$\mathcal{D}_g \in \mathbb{C}^{a \times b \times r}$ denotes a sub-tensor of $\mathcal{D}$, with indices $t=1,\ldots,a$, \ $g=1,\ldots,p$ and $j=1,\ldots,b$.
Furthermore, the tensor $\mathcal{B}$ possesses the block representation given below:
$$ \mathcal{B} =\left[
\begin{array}{c}
\operatorname{Row}_1(\mathcal{B}) \\
\operatorname{Row}_2(\mathcal{B}) \\
\vdots \\
\operatorname{Row}_{\frac{h}{m}}(\mathcal{B}) \\ \hdashline
\vdots \\ \hdashline
\operatorname{Row}_{h-\frac{h}{m}+1}(\mathcal{B}) \\
\operatorname{Row}_{h-\frac{h}{m}+2}(\mathcal{B}) \\
\vdots \\
\operatorname{Row}_h(\mathcal{B})
\end{array}
\right]=
\left[
\begin{array}{c}
\operatorname{Block}_1(\mathcal{B}) \\
\vdots \\
\operatorname{Block}_m(\mathcal{B})
\end{array}
\right]
$$\\
where
$$
\operatorname{Block}_s(B) =
\begin{bmatrix}
\operatorname{Row}_{\frac{(s-1)h}{m}+1}(\mathcal{B}) \\
\vdots \\
\operatorname{Row}_{\frac{sh}{m}}(\mathcal{B})
\end{bmatrix}
$$
is a Toeplitz tensor and $s=1,...,m.$
\end{theorem}
\begin{proof}
Based on Lemma~\ref{lemma:3.4}, we consider the case where $X\in \mathbb{C}^P$ satisfies Eq\eqref{eq:3.2}. with $p = \frac{nh}{mk} = \frac{l}{a}$. Since $\frac{h}{m}$ and $k$ are relatively prime, we can express $k$ as $k = l_1 \cdot \frac{h}{m} + l_2$ for some integers $l_1$ and $l_2$. Then, Eq\eqref{eq:3.1} can be rewritten as
\begin{center}
$\operatorname{Row}_s(\mathcal{A})\ltimes X=x_{1}\mathcal{A}_{(::i)}^1+x_2\mathcal{A}_{(::i)}^2+\cdots+x_{\frac{h}{m}}\mathcal{A}_{(::i)}^{\frac{h}{m}}+x_{\frac{h}{m}+1}\mathcal{A}_{(::i)}^{\frac{h}{m}+1}+\cdots+x_p\mathcal{A}_{::i}^{p}=\text{Block}_s(\mathcal{B})\in \mathbb{C}^{\frac{h}{m}\times k\times r},$
\end{center}
 which is a Toeplitz tensor.
 Where
 \begin{align*}
&\mathcal{A}^{1}_{:,:,i}= \left[
\begin{array}{*{14}c} 
   a_{s,1,i}& & & & & & a_{s,l_{1},i} & & & & & a_{s,(l_1+1),i}& &\\
   &\ddots &  & & & & &\ddots & & & & &\ddots & \\
   &&a_{s,1,i}& & & \cdots & & & a_{s,l_{1},i} & & & & & a_{s,(l_1+1),i}\\
   & & &\ddots & & & & & &\ddots & & & &\\
   & & & & a_{s,1,i} & & & & & & a_{s,l_{1},i}& & &\\
\end{array}
\right],\\
&\mathcal{A}^{2}_{:,:,i}=\left[
\begin{array}{*{12}c}
& & &  a_{s,(l_1+2),i} & & & & &  &a_{s,\left[\frac{2k}{h/m}\right]+1,i}& &\\
& & & & \ddots & & & & && \ddots &\\
 a_{s,(l_1+1),i}& & & & &  a_{s,(l_1+2),i} & & &\cdots  & & & a_{s,\left[\frac{2k}{h/m}\right]+1,i}\\
 &\ddots& & & & &\ddots & & & & &\\
 & & a_{s,(l_1+1),i}& & & & &  a_{s,(l_1+2),i} & & & &
 \end{array}
\right],\\
&\vdots
\end{align*}
\begin{align*}
&\mathcal{A}^{\frac{h}{m}}_{:,:,i}=\left[
\begin{array}{*{9}c} 
    & & & &a_{s,k,i}& & & &\\
    & & & & &\ddots & & &\\
    a_{s,k-l_1,i}& & &\cdots& & & a_{s,k,i}& &\\
    &\ddots & & & & & &\ddots &\\
   & & a_{s,k-l_1,i} & & & & & & a_{s,k,i}
\end{array}
\right],\\
&\mathcal{A}^{\frac{h}{m}+1}_{:,:,i}=\\ 
& \small{\left[
\begin{array}{*{11}c} 
 a_{s,k+1,i}& & & & && a_{s,k+l_{1},i}& & & & \\
 &\ddots& & & & &&\ddots& & &  \\
 & &a_{s,k+1,i}& & &\cdots & &&a_{s,k+l_{1},i}& &\\
 & & &\ddots& & & & &&\ddots& \\
 & & & &a_{s,k+1,i}& & & & & &a_{s,k+l_{1},i}\\
\end{array}
\begin{array}{*{4}c}
     a_{s,k+l_{1}+1,i} & & &  \\
     & \ddots&&\\
     &&& a_{s,k+l_{1}+1,i}
\end{array}
\right],}\\
&\vdots\\
&\mathcal{A}^{p}_{:,:,i}=\left[
\begin{array}{*{9}c} 
   & & & &a_{s,n,i}& & & &\\
    & & & & &\ddots & & &\\
    a_{s,n-l_1,i}& & &\cdots& & & a_{s,n,i}& &\\
    &\ddots & & & & & &\ddots &\\
   & & a_{s,n-l_1,i} & & & & & & a_{s,n,i}
\end{array}
\right],
\end{align*}
 Then, by Eq\eqref{eq:3.5} we have 
 $$\operatorname{Row}_g(X) \ltimes \mathcal{C}
= x_g
\begin{bmatrix}
c_{1,1,i} & c_{1,2,i} & \cdots & c_{1,b,i} \\
c_{2,1,i} & c_{2,2,i} & \cdots & c_{2,b,i} \\
\vdots  & \vdots  & \ddots & \vdots  \\
c_{a,1,i} & c_{a,2,i} & \cdots & c_{a,b,i}
\end{bmatrix}
= \mathcal{D}_g
=
\begin{bmatrix}
d_{(g-1)a+1,1,i} & d_{(g-1)a+1,2,i} & \cdots & d_{(g-1)a+1,b,i} \\
d_{(g-1)a+2,1,i} & d_{(g-1)a+2,2,i} & \cdots & d_{(g-1)a+2,b,i} \\
\vdots & \vdots & \ddots & \vdots \\
d_{ga,1,i} & d_{ga,2,i} & \cdots & d_{ga,b,i}
\end{bmatrix}.$$
Hence, we have
$d_{\,t + (g-1)a,\, j,\, i} = \alpha_g \, c_{t,j,i}.$
This completes the proof.
 \end{proof}
The following derivations will be used in the rest of this section.
$$\begin{aligned}
& x_{1}
\begin{bmatrix}
a_{11,i} & a_{12,i} & \cdots & a_{1,l_1,i} & a_{1,l_1+1,i} \\
a_{21,i} & a_{22,i} & \cdots & a_{2,l_1,i} & a_{2,l_1+1,i} \\
\vdots & \vdots & \ddots & \vdots    & \vdots \\
a_{m1,i} & a_{m2,i} & \cdots & a_{m,l_1,i} & a_{m,l_1+1,i}
\end{bmatrix}
+\, x_{{h}/{m}+1}
\begin{bmatrix}
a_{1,k+1,i} & a_{1,k+2,i} & \cdots & a_{1,k+l_1,i} & a_{1,k+l_1+1,i} \\
a_{2,k+1,i} & a_{2,k+2,i} & \cdots & a_{2,k+l_1,i} & a_{2,k+l_1+1,i} \\
\vdots & \vdots & \ddots & \vdots & \vdots \\
a_{m,k+1,i} & a_{m,k+2,i} & \cdots & a_{m,k+l_1,i} & a_{m,k+l_1+1,i}
\end{bmatrix} \\[1em]
&\quad
+ \cdots +
x_{\left(\frac{nb}{k}-1\right){h}/{m}+1}
\begin{bmatrix}
a_{1,n-k+1,i} & a_{1,n-k+2,i} & \cdots & a_{1,n-k+l_1,i} & a_{1,n-k+l_1+1,i} \\
a_{2,n-k+1,i} & a_{2,n-k+2,i} & \cdots & a_{2,n-k+l_1,i} & a_{2,n-k+l_1+1,i} \\
\vdots & \vdots & \ddots & \vdots & \vdots \\
a_{m,n-k+1,i} & a_{m,n-k+2,i} & \cdots & a_{m,n-k+l_1,i} & a_{m,n-k+l_1+1,i}
\end{bmatrix} \\[1em]
&=
\begin{bmatrix}
b_{1,1,i} & b_{1,{h}/{m}+1,i} & \cdots & b_{1,(l_1-1){h}/{m}+1,i} & b_{1,l_1{h}/{m}+1,i} \\
b_{{h}/{m}+1,1,i} & b_{{h}/{m}+1,{h}/{m}+1,i} & \cdots & b_{{h}/{m}+1,(l_1-1){h}/{m}+1,i} & b_{{h}/{m}+1,l_1{h/}{m}+1,i} \\
\vdots & \vdots & \ddots & \vdots & \vdots \\
b_{h-{h}/{m}+1,1,i} & b_{h-{h}/{m}+1,{h}/{m}+1,i} & \cdots &
b_{h-{h}/{m}+1,(l_1-1){h}/{m}+1,i} &
b_{h-{h}/{m}+1,l_1{h}/{m}+1,i}
\end{bmatrix},\\
\end{aligned}$$
$\begin{aligned}
& x_{2}
\begin{bmatrix}
a_{1,l_1+1,i} & a_{1,l_1+2,i} & \cdots & a_{1,\lfloor 2k/(h/m)\rfloor,i} & a_{1,\lfloor 2k/(h/m)\rfloor+1,i} \\
a_{2,l_1+1,i} & a_{2,l_1+2,i} & \cdots & a_{2,\lfloor 2k/(h/m)\rfloor,i} & a_{2,\lfloor 2k/(h/m)\rfloor+1,i} \\
\vdots & \vdots & \ddots & \vdots & \vdots \\
a_{m,l_1+1,i} & a_{m,l_1+2,i} & \cdots & a_{m,\lfloor 2k/(h/m)\rfloor,i} & a_{m,\lfloor 2k/(h/m)\rfloor+1,i}
\end{bmatrix}+\, \\ 
& x_{h/m+2} 
\begin{bmatrix}
a_{1,k+l_1+1,i} & a_{1,k+l_1+2,i} & \cdots & a_{1,k+\lfloor 2k/(h/m)\rfloor,i} & a_{1,k+\lfloor 2k/(h/m)\rfloor+1,i} \\
a_{2,k+l_1+1,i} & a_{2,k+l_1+2,i} & \cdots & a_{2,k+\lfloor 2k/(h/m)\rfloor,i} & a_{2,k+\lfloor 2k/(h/m)\rfloor+1,i} \\
\vdots & \vdots & \ddots & \vdots & \vdots \\
a_{m,k+l_1+1,i} & a_{m,k+l_1+2,i} & \cdots & a_{m,k+\lfloor 2k/(h/m)\rfloor,i} & a_{m,k+\lfloor 2k/(h/m)\rfloor+1,i}
\end{bmatrix}
\\
\end{aligned}$
$$
\begin{aligned}
+ \cdots +
x_{(n/k-1)h/m+2,1}
\begin{bmatrix}
a_{1,n-k+l_1+1,i} & a_{1,n-k+l_1+2,i} & \cdots &
a_{1,n-k+\lfloor 2k/(h/m)\rfloor,i} &
a_{1,n-k+\lfloor 2k/(h/m)\rfloor+1,i} \\
a_{2,n-k+l_1+1,i} & a_{2,n-k+l_1+2,i} & \cdots &
a_{2,n-k+\lfloor 2k/(h/m)\rfloor,i} &
a_{2,n-k+\lfloor 2k/(h/m)\rfloor+1,i} \\
\vdots & \vdots & \ddots & \vdots & \vdots \\
a_{m,n-k+l_1+1,i} & a_{m,n-k+l_1+2,i} & \cdots &
a_{m,n-k+\lfloor 2k/(h/m)\rfloor,i} &
a_{m,n-k+\lfloor 2k/(h/m)\rfloor+1,i}
\end{bmatrix}
\\[1em]
\end{aligned}$$
$$\begin{aligned}
&=\begin{bmatrix}
b_{l_2+1,1,i} & b_{1,h/m-l_2+1,i} & \cdots &
b_{1,(\lfloor 2k/(h/m)\rfloor-l_1-1)h/m-l_2+1,i} \\
b_{h/m+1,l_2+1,i} & b_{h/m+1,h/m-l_2+1,i} & \cdots &
b_{h/m+1,(\lfloor 2k/(h/m)\rfloor-l_1-1)h/m-l_2+1,i} \\
\vdots & \vdots & \ddots & \vdots \\
b_{h-h/m+1,l_2+1,i} & b_{h-h/m+1,h/m-l_2+1,i} & \cdots &
b_{h-h/m+1,(\lfloor 2k/(h/m)\rfloor-l_1-1)h/m-l_2+1,i}
\end{bmatrix},\\
\vdots & \\
& x_{h/m}
\begin{bmatrix}
a_{1,k-l_1,i} & a_{1,k-l_1+1,i} & \cdots & a_{1,k-1,i} & a_{1,k,i} \\
a_{2,k-l_1,i} & a_{2,k-l_1+1,i} & \cdots & a_{2,k-1,i} & a_{2,k,i} \\
\vdots & \vdots & \ddots & \vdots & \vdots \\
a_{m,k-l_1,i} & a_{m,k-l_1+1,i} & \cdots & a_{m,k-1,i} & a_{m,k,i}
\end{bmatrix}
+ x_{2h/m}
\begin{bmatrix}
a_{1,2k-l_1,i} & a_{1,2k-l_1+1,i} & \cdots & a_{1,2k-1,i} & a_{1,2k,i} \\
a_{2,2k-l_1,i} & a_{2,2k-l_1+1,i} & \cdots & a_{2,2k-1,i} & a_{2,2k,i} \\
\vdots & \vdots & \ddots & \vdots & \vdots \\
a_{m,2k-l_1,i} & a_{m,2k-l_1+1,i} & \cdots & a_{m,2k-1,i} & a_{m,2k,i}
\end{bmatrix}
\nonumber \\[1ex]
&+ \cdots
+ x_{p}
\begin{bmatrix}
a_{1,n-l_1,i} & a_{1,n-l_1+1,i} & \cdots & a_{1,n-1,i} & a_{1,n,i} \\
a_{2,n-l_1,i} & a_{2,n-l_1+1,i} & \cdots & a_{2,n-1,i} & a_{2,n,i} \\
\vdots & \vdots & \ddots & \vdots & \vdots \\
a_{m,n-l_1,i} & a_{m,n-l_1+1,i} & \cdots & a_{m,n-1,i} & a_{m,n,i}
\end{bmatrix}
\nonumber \\[2ex]
\end{aligned}$$
$$\begin{aligned}
&~=
\begin{bmatrix}
b_{1,h/m-l_2+1,i} & b_{1,l_2+1,i} & \cdots & b_{1,(l_1-2)h/m+l_2+1,i} & b_{1,k-h/m+1,i} \\
b_{2h/m-l_2+1,1,i} & b_{h/m+1,l_2+1,i} & \cdots & b_{h/m+1,(l_1-2)h/m+l_2+1,i} & b_{h/m+1,k-h/m+1,i} \\
\vdots & \vdots & \ddots & \vdots & \vdots \\
b_{h-l_2+1,1,i} & b_{h-h/m+1,l_2+1,i} & \cdots & b_{h-h/m+1,(l_1-2)h/m+l_2+1,i} & b_{h-h/m+1,k-h/m+1,i}
\end{bmatrix},
\end{aligned} $$
$$\begin{aligned}
&x_1
\begin{bmatrix}
c_{1,1,i} & c_{1,2,i} & \cdots & c_{1,b,i} \\
c_{2,1,i} & c_{2,2,i} & \cdots & c_{2,b,i} \\
\vdots  & \vdots  & \ddots & \vdots  \\
c_{a,1,i} & c_{a,2,i} & \cdots & c_{a,b,i}
\end{bmatrix}
=\begin{bmatrix}
d_{1,1,i} & d_{1,2,i} & \cdots & d_{1,b,i} \\
d_{2,1,i} & d_{2,2,i} & \cdots & d_{2,b,i} \\
\vdots & \vdots & \ddots & \vdots \\
d_{a,1,i} & d_{a,2,i} & \cdots & d_{a,b,i}
\end{bmatrix},\\
&x_2\begin{bmatrix}
c_{1,1,i} & c_{1,2,i} & \cdots & c_{1,b,i} \\
c_{2,1,i} & c_{2,2,i} & \cdots & c_{2,b,i} \\
\vdots  & \vdots  & \ddots & \vdots  \\
c_{a,1,i} & c_{a,2,i} & \cdots & c_{a,b,i}
\end{bmatrix}
=\begin{bmatrix}
d_{a+1,1,i} & d_{a+1,2,i} & \cdots & d_{a+1,b,i} \\
d_{a+2,1,i} & d_{a+2,2,i} & \cdots & d_{a+2,b,i} \\
\vdots & \vdots & \ddots & \vdots \\
d_{2a,1,i} & d_{2a,2,i} & \cdots & d_{2a,b,i}
\end{bmatrix},
\\
&\vdots\\
&x_p
\begin{bmatrix}
c_{1,1,i} & c_{1,2,i} & \cdots & c_{1,b,i} \\
c_{2,1,i} & c_{2,2,i} & \cdots & c_{2,b,i} \\
\vdots  & \vdots  & \ddots & \vdots  \\
c_{a,1,i} & c_{a,2,i} & \cdots & c_{a,b,i}
\end{bmatrix}
=\begin{bmatrix}
d_{l-a+1,1,i} & d_{l-a+1,2,i} & \cdots & d_{l-a+1,b,i} \\
d_{l-a+2,1,i} & d_{l-a+2,2,i} & \cdots & d_{l-a+2,b,i} \\
\vdots & \vdots & \ddots & \vdots \\
d_{l,1,i} & d_{l,2,i} & \cdots & d_{l,b,i}
\end{bmatrix}.
\end{aligned}$$
For simplicity, we focus on the case
$r \cdot k = l_1^{(r)} \cdot \frac{h}{m} + l_2^{(r)}, \quad r = 1, \ldots, \frac{h}{m}$,
which permits the equations to be expressed in the simplified form below
\[
\left\{
\begin{aligned}
&x_1\!\left[\mathcal{A}_1 \; \mathcal{A}_2 \; \cdots \; \mathcal{A}_{l_1} \; \mathcal{A}_{l_1+1}\right]
 + x_{\frac{h}{m}+1}\!\left[\mathcal{A}_{k+1} \; \mathcal{A}_{k+2} \; \cdots \; \mathcal{A}_{k+l_1} \; \mathcal{A}_{k+l_1+1}\right] \\
&\quad + \cdots
 + x_{\left(\frac{n}{k}-1\right)\frac{h}{m}+1}
 \!\left[\mathcal{A}_{n-k+1} \; \mathcal{A}_{n-k+2} \; \cdots \; \mathcal{A}_{n-k+l_1} \; \mathcal{A}_{n-k+l_1+1}\right] \\
&\qquad =
\left[\widetilde{\mathcal{B}}_1 \;
      \widetilde{\mathcal{B}}_{\frac{h}{m}+1} \;
      \cdots \;
      \widetilde{\mathcal{B}}_{(l_1-1)\frac{h}{m}+1} \;
      \widetilde{\mathcal{B}}_{l_1\frac{h}{m}+1}\right], \\[1ex]
&x_2\!\left[\mathcal{A}_{l_1+1} \; \mathcal{A}_{l_1+2} \; \cdots \; \mathcal{A}_{2l_1} \; \mathcal{A}_{2l_1+1}\right]
 + x_{\frac{h}{m}+2}\!\left[\mathcal{A}_{k+l_1+1} \; \mathcal{A}_{k+l_1+2} \; \cdots \; \mathcal{A}_{k+2l_1} \; \mathcal{A}_{k+2l_1+1}\right] \\
&\quad + \cdots
 + x_{\left(\frac{n}{k}-1\right)\frac{h}{m}+2}
 \!\left[\mathcal{A}_{n-k+l_1+1} \; \mathcal{A}_{n-k+l_1+2} \; \cdots \; \mathcal{A}_{n-k+2l_1} \; \mathcal{A}_{n-k+2l_1+1}\right] \\
&\qquad =
\left[\widetilde{\mathcal{B}}_{k_2+1} \;
      \widetilde{\mathcal{B}}_{\frac{h}{m}-l_2+1} \;
      \cdots \;
      \widetilde{\mathcal{B}}_{(l_1-1)\frac{h}{m}-l_2+1} \;
      \widetilde{\mathcal{B}}_{l_1\frac{h}{m}-l_2+1}\right], \\[1ex]
&\vdots \\[1ex]
&x_{\frac{h}{m}}\!\left[\mathcal{A}_{k-l_1} \; \mathcal{A}_{k-l_1+1} \; \cdots \; \mathbb{A}_{k-1} \; \mathcal{A}_k\right]
 + x_{2\frac{h}{m}}\!\left[\mathcal{A}_{2k-l_1} \; \mathcal{A}_{2k-l_1+1} \; \cdots \; \mathcal{A}_{2k-1} \; \mathcal{A}_{2k}\right] \\
&\quad+\cdots
 + x_p\!\left[\mathcal{A}_{n-l_1} \; \mathcal{A}_{n-l_1+1} \; \cdots \; \mathcal{A}_{n-1} \; \mathcal{A}_n\right] \\
&\qquad =
\left[\widetilde{\mathcal{B}}_{k+\frac{h}{m}-l_2} \;
      \widetilde{\mathcal{B}}_{l_2+1} \;
      \cdots \;
      \widetilde{\mathcal{B}}_{(l_1-2)\frac{h}{m}+l_2+1} \;
      \widetilde{\mathcal{B}}_{k-\frac{h}{m}+1}\right], \\[1ex]
&x_1\,[\mathcal{C}_1 \; \mathcal{C}_2 \; \cdots \; \mathcal{C}_a]^T
 = [\mathcal{D}_1 \; \mathcal{D}_2 \; \cdots \; \mathcal{D}_a]^T, \\
&x_2\,[\mathcal{C}_1 \; \mathcal{C}_2 \; \cdots \; \mathcal{C}_a]^T
 = [\mathcal{D}_{a+1} \; \mathcal{D}_{a+2} \; \cdots \; \mathcal{D}_{2a}]^T, \\
&\vdots \\
&x_p\,[\mathcal{C}_1 \; \mathcal{C}_2 \; \cdots \; \mathcal{C}_a]^T
 = [\mathcal{D}_{(p-1)a+1} \; \mathcal{D}_{(p-1)a+2} \; \cdots \; \mathcal{D}_{pa}]^T,
 \end{aligned}
\right.
\]
where
$$
\widetilde{\mathcal{B}} =
\begin{bmatrix}
b_{1,1,i} & b_{1,2,i} & \cdots & b_{1,k,i} & b_{2,1,i} & b_{3,1,i} & \cdots & b_{h/m,1,i} \\
b_{h/m+1,1,i} & b_{h/m+1,2,i} & \cdots & b_{h/m+1,k,i} & b_{h/m+2,1,i} & b_{h/m+3,1,i} & \cdots & b_{2h/m,1,i} \\
\vdots & \vdots & \ddots & \vdots & \vdots & \vdots & \ddots & \vdots \\
b_{h-h/m+1,1,i} & b_{h-h/m+1,2,i} & \cdots & b_{h-h/m+1,k,i} & b_{h-h/m+2,1,i} & b_{h-h/m+3,1,i} & \cdots & b_{h,1,i}
\end{bmatrix}.
$$\\
Here, $\widetilde{\mathcal{B}}_j$ denotes the $j^{th}$ lateral slice of $\mathcal{B}$. In addition, $\mathcal{C}_j$ and $\mathcal{D}_t$ correspond to the $j^{th}$ and $t^{th}$ horizontal slices of the tensors $\mathcal{C}$ and $\mathcal{D}$, respectively.\\
From above discussion we have the following result.
\begin{theorem}\label{theorem3.6}
Solvability of tensor vector Eq\eqref{eq:3.1} is equivalent to the solvability of the tensor-vector equations described in the following system:
\begin{equation}
\left\{
\begin{aligned}
&[\widehat{\mathcal{A}}_1 \ \widehat{\mathcal{A}}_{\frac{h}{m}+1} \ \cdots \ \widehat{\mathcal{A}}_{(\frac{n}{k}-1)\frac{h}{m}+1}] \ltimes Y_1 = \bar{\mathcal{B}}_1,\\
&[\widehat{\mathcal{A}}_2 \ \widehat{\mathcal{A}}_{\frac{h}{m}+2} \ \cdots \ \widehat{\mathcal{A}}_{(\frac{n}{k}-1)\frac{h}{m}+2}] \ltimes Y_2 = \bar{\mathcal{B}}_2,\\
&\qquad \vdots \\
&[\widehat{\mathcal{A}}_{\frac{h}{m}} \ \widehat{\mathcal{A}}_{\frac{2h}{m}} \ \cdots \ \widehat{\mathcal{A}}_{p}] \ltimes Y_{\frac{h}{m}} = \bar{\mathcal{B}}_\frac{h}{m},\\
&x_1 \mathcal{C} = \mathcal{D}_1 ,\\
&\qquad \vdots \\
&x_p \mathcal{C} = \mathcal{D}_p,
\end{aligned}
\right.
\end{equation}
\vspace{-1em}
\noindent
\begin{align*}
\text{where } &\mathcal{A} =
[
\underbrace{\mathcal{A}_1 \ \cdots \ \mathcal{A}_{l_1^1}\  \mathcal{A}_{l_1^1+1}}_{\widehat{\mathcal{A}}_{1}} \ 
\ \cdots \ \mathcal{A}_{l_1^2}\ \mathcal{A}_{l_1^2} \ 
\cdots \ 
\underbrace{\mathcal{A}_{n-l_1^1}\ \mathcal{A}_{n-l_1^1+1} \ \cdots \ A_{n}}_{\widehat{\mathcal{A}}_p}
],
\end{align*}
\begin{align*}
&\overline{\mathcal{B}}_1 =
[
\widetilde{\mathcal{B}}_1 \ 
\widetilde{\mathcal{B}}_{\frac{h}{m}+1} \ 
\cdots \ 
\widetilde{\mathcal{B}}_{(l_1^{1}-1)\frac{h}{m}+1} \ 
\widetilde{\mathcal{B}}_{l_1^1\frac{h}{m}+1}
],\\
&\overline{\mathcal{B}}_2 =
[
\widetilde{\mathcal{B}}_{k+l_2^1} \ 
\widetilde{\mathcal{B}}_{\frac{h}{m}-l_2^1+1} \ 
\cdots \ 
\widetilde{\mathcal{B}}_{(l_2^1-l_1^1-1)\frac{h}{m}-l_2^1+1} \ 
\widetilde{\mathcal{B}}_{(l_1^2-l_1^1)\frac{h}{m}-l_2^1+1}
],\\
& \vdots\\
&\overline{\mathcal{B}}_\frac{h}{m} =
[
\widetilde{\mathcal{B}}_{k+\frac{h}{m}-l_1^1} \ 
\widetilde{\mathcal{B}}_{l_2^1+1} \ 
\cdots \ 
\widetilde{\mathcal{B}}_{(l_1^1-2)+l_2^1+1} \ 
\widetilde{\mathcal{B}}_{k-\frac{h}{m}+1}
],
\end{align*}
and
$Y = [ Y_1 \ Y_2 \ \cdots \ Y_u \ \cdots \ Y_{\frac{h}{m}} ]^T$,
where
$Y_u = [ y_{u1} \ y_{u2} \ \cdots \ y_{u\frac{n}{k}} ]^T,
\quad u = 1,2,\ldots,\frac{h}{m}.$
Consequently, the vector 
$X$ can be expressed as
\[
X = V_c(Y).
\]
\end{theorem}
Furthermore, an additional result regarding the existence of the solution of tensor-vector Eq~\eqref{eq:3.1} can be derived from Theorem~\ref{theorem3.6}.
\begin{theorem}\label{theorem:3.7}
Eq\eqref{eq:3.1} admits a solution if and only if there exists a vector $X$ such that the resulting two conditions hold:\vspace{-0.5em}
\begin{enumerate}
    \item[(i)] For each $r = 1, 2, \ldots, \frac{h}{m}$, the tensors 
    $\widehat{\mathcal{A}}_r, \widehat{\mathcal{A}}_{\frac{h}{m}+r}, \ldots, \widehat{\mathcal{A}}_{\left(\frac{n}{k}-1\right)\frac{h}{m}+r}$ and $\bar{\mathcal{B}}_r$ are linearly dependent.\vspace{-0.7em}
    \item[(ii)] The entries of $\mathcal{D}$ satisfy
   $ d_{\,t + (g-1)a, j, i} = \alpha_g \, c_{t,j,i}$\ ,
    where $d_{\,t + (g-1)a, j, i} \in \mathcal{D}_g$, $g = 1, \ldots, p$ denotes the $g$-th block of $\mathcal{D}$, $t = 1, \ldots, a$, and $j = 1, \ldots, d$.
\end{enumerate}\vspace{-0.4em}
In this case, the solution vector is
$X = [\alpha_1, \ldots, \alpha_p]^T$
Moreover, the solution is unique unless $\mathcal{C} = \mathcal{D} = 0$ and the tensors
$\widehat{\mathcal{A}}_r, \widehat{\mathcal{A}}_{\frac{h}{m}+r}, \ldots, \widehat{\mathcal{A}}_{\left(\frac{k}{m}-1\right)\frac{h}{m}+r}$
remain linearly dependent.
\end{theorem}

Next, we provide example illustrating the above theorem \ref{eq:3.5}, Theorem \ref{theorem3.6}.\vspace{-0.5em}
\begin{example} and Theorem \ref{theorem:3.7}.
Let $\mathcal{A,\ B,\ C}$ and $\mathcal{D}$ be t order tensors defined as  \begin{center}
    \begin{tabular}{c c c|c c c}
\hline
  \multicolumn{3}{c}{$\mathcal{A}(:,:,1)$} & 
  \multicolumn{3}{c}{$\mathcal{A}(:,:,2)$} \\
  
\hline
2&1&-1&4&2&-2\\
3&2&1&1&4&2\\
\hline
\end{tabular}, \ 
\begin{tabular}{ccc|ccc}
\hline
  \multicolumn{3}{c}{$\mathcal{B}(:,:,1)$} & 
  \multicolumn{3}{c}{$\mathcal{B}(:,:,2)$} \\
  \hline
    4 &1&2&8&2&4  \\
     -1&4&1&-2&8&2\\
     6&-1&4&2&-2&8\\
     -2&6&-1&-4&2&-2
\end{tabular},\quad
    \begin{tabular}{c c|c c}
    \hline
  \multicolumn{2}{c}{$\mathcal{C}(:,:,1)$} & 
  \multicolumn{2}{c}{$\mathcal{C}(:,:,2)$} \\
  \hline
         3&1&-1&1  \\
         1&3&0&-1 
    \end{tabular},\quad
\begin{tabular}{cc|cc}
\hline
  \multicolumn{2}{c}{$\mathcal{D}(:,:,1)$} & 
  \multicolumn{2}{c}{$\mathcal{D}(:,:,2)$} \\
  
\hline
    6 &2&-2&2  \\
     2&6&0&-2\\
     -3&-1&1&-1\\
     -1&-3&0&1
\end{tabular}. \
\end{center}
Here $\frac{h}{m}=2,\ \frac{n}{k}=1 ,\ \frac{l}{a}=1$ and $k=3=1.2+1,\ p=2.$ Then we have \begin{center}
    \begin{tabular}{cc|cc}
    \hline
  \multicolumn{2}{c}{$\widehat{\mathcal{A}}_{1}(:,:,1)$} & 
  \multicolumn{2}{c}{$\widehat{\mathcal{A}}_{1}(:,:,2)$} \\
  \hline
        2 &1&4&2  \\
         3&2&1&4 
    \end{tabular}, \quad
    \begin{tabular}{cc|cc}
    \hline
  \multicolumn{2}{c}{$\widehat{\mathcal{A}}_{2}(:,:,1)$} & 
  \multicolumn{2}{c}{$\widehat{\mathcal{A}}_{2}(:,:,2)$} \\
  \hline
        1 &-1&2&-2  \\
         2&1&4&2 
    \end{tabular}, \quad 
    \begin{tabular}{cccc|cccc}
    \hline
  \multicolumn{4}{c}{$\widetilde{\mathcal{B}}(:,:,1)$} & 
  \multicolumn{4}{c}{$\widetilde{\mathcal{B}}(:,:,2)$}\\
  \hline
        4&1&2&-1&8&2&4&-2  \\
         6&-1&4&-2&2&-2&8&4
    \end{tabular},\quad 
\end{center}
\begin{center}
\begin{tabular}{cc|cc}
    \hline
  \multicolumn{2}{c}{$\overline{\mathcal{B}}_{1}(:,:,1)$} & 
  \multicolumn{2}{c}{$\overline{\mathcal{B}}_{1}(:,:,2)$} \\
  \hline
        4&2&8&4 \\
         6&4&2&8 
    \end{tabular}, \quad
    \begin{tabular}{cc|cc}
    \hline
  \multicolumn{2}{c}{$\overline{\mathcal{B}}_{2}(:,:,1)$} & 
  \multicolumn{2}{c}{$\overline{\mathcal{B}}_{2}(:,:,2)$} \\
  \hline
        -1&1&-2&2 \\
         -2&-1&-4&-2
    \end{tabular}, \quad 
    \begin{tabular}{cc|cc}
    \hline
  \multicolumn{2}{c}{${\mathcal{D}}_{1}(:,:,1)$} & 
  \multicolumn{2}{c}{${\mathcal{D}}_{1}(:,:,2)$} \\
  \hline
        6&2&-2&2 \\
         2&6&0&-2 
    \end{tabular}, \quad
     \begin{tabular}{cc|cc}
    \hline
  \multicolumn{2}{c}{${\mathcal{D}}_{2}(:,:,1)$} & 
  \multicolumn{2}{c}{${\mathcal{D}}_{2}(:,:,2)$} \\
  \hline
        -3&-1&1&-1 \\
         -1&-3&0&1
\end{tabular}
\end{center}
Moreover, the tensors $\widehat{\mathcal{A}}_u$ and $\bar{\mathcal{B}}_u$ are linearly dependent, and $\mathcal{C}$ and $\mathcal{D}_g$ satisfy multiple dependencies for $u = g = 1, 2$. Therefore, by Theorem~\ref{theorem3.6}, the unique solution is
$Y_1 = 2, \quad Y_2 = -1, \quad X = [2, -1]^T.$
\end{example}
Based on the definition of a circulant matrix, we now introduce its tensor analogue, referred to as a circulant tensor.
\begin{definition}\label{toeplitz def}
(Circulant tensor) Denote $[n]=\{1,2,\ldots,n\}.$ Let $\mathcal{A} = (a_{i_1 \cdots i_m})$ be a real $m^{th}$ order $n$-dimensional tensor. If for $i_l, k_l \in [n]$ satisfying 
$k_l \equiv i_l + 1 \pmod{n}, \ l \in [m],$ we have
$$a_{i_1 \cdots i_{m-1} i_m} = a_{k_1 \cdots k_{m-1} i_m},$$
then we say that $\mathcal{A}$ is an $m^{th}$ order circulant tensor. $\mathcal{A} := \operatorname{circ}(\mathcal{A}),$
is referred as Circulant tensor
whenever all of its frontal slices are Circulant matrices.
\end{definition}
\begin{definition}(F-diagonal tensor \cite{M.Kilmer2011factorizationfortensor})
A third order tensor $\mathcal{D}\in \mathbb{C}^{n\times n\times  r}$ is F-diagonal, if it's frontal slices are all diagonal. 
\end{definition}
\begin{remark}
    Assume that the tensor–vector Eq\eqref{eq:3.1} has a solution. If $\mathcal{A}$ is a F-diagonal tensor with identical diagonal entries, then each sub-tenosr of $\mathcal{B}$ exhibits a Circulant tensor structure.
\end{remark}\vspace{-0.5em}
\section{Solvability conditions for tensor equations with matrix solutions }
In this part, we analyze the existence of solution of the resulting tensor-matrix equation using the STP:
\begin{equation}\label{eq:4.1}
\left\{
\begin{aligned}
& \underbrace{\mathcal{A} \ltimes{X}}_{\frac{m t_1}{n} \times \frac{qt_1}{p} \times d_1} = \mathcal{\mathcal{B}}, \\
& \underbrace{{X} \ltimes \mathcal{C}}_{\frac{p t_{2}}{q} \times \frac{b t_{2}}{a} \times d_2} = \mathcal{D},
\end{aligned}
\right.
\end{equation}
here, $\mathcal{A} \in \mathbb{C}^{m \times n \times r}$, $\mathcal{B} \in \mathbb{C}^{h \times k \times r}$, $\mathcal{C} \in \mathbb{C}^{a \times b \times r}$, and $\mathcal{D} \in \mathbb{C}^{l \times b \times r}$ represent known tensors, and $X \in \mathbb{C}^{p \times q}$ represents the unknown matrix. The parameters are defined as $t_1 =[n, p]$, $t_2 = [q, a]$, and $d_1 = d_2 = [1, r]$. First, the simplified case with $m = h$ and $l = p$ is considered, followed by the general scenario.
\subsection{The simplified case with \texorpdfstring{$m=h$}{m=h} and \texorpdfstring{$l=p$}{l=p}}
The solvability of the tensor-matrix Eq\eqref{eq:4.1} is analyzed in this section. The following lemma is obtained by applying the STP.
\begin{lemma}\label{lemma:4.1}
Let $X \in \mathbb{C}^{p \times q}$ be a solution of the tensor-matrix Eq\eqref{eq:4.1}. Then, the dimensions of the tensors $\mathcal{A}, \mathcal{B}, \mathcal{C}$, and $\mathcal{D}$ necessarily hold the following conditions:\vspace{-0.5em}
\begin{enumerate}
    \item[(i)] The ratios $\frac{n}{l}$ and $\frac{d}{b}$ are positive integers; (ii) The ratios $\frac{ad}{b}$ and $\frac{n}{l}$ divide $k$, and $q = \frac{ad}{b} = \frac{l k}{n}$.
\end{enumerate}
\end{lemma}
\begin{proof}
 Let $X\in \mathbb{C}^{p\times q}$ satisfy Eq\eqref{eq:4.1} and from Eq\eqref{eq:4.1} we get
$$\frac{mt_{1}}{n}=h, \frac{qt_{1}}{p}=k,\quad \frac{pt_{2}}{q}=l, \frac{bt_{2}}{a}=d,$$
\vspace{-0.3em}
\noindent where $ t_{1}=[n, p] \text{ and }d_{1}=[r,1]=r ,\ t_{2}=[q, a] \ \text{ and } \ d_{2}=[r,1]=r .$\\
Hence, $\frac{t_1}{n} = \frac{h}{m}$ is a positive integer. Under the assumptions $m = h$ and $l = p$, it follows that $t_1 = n$, $t_2 = q$, and$
\frac{t_2}{a} = \frac{d}{b} = \frac{q}{a}, \ \frac{t_1}{p} = \frac{n}{l} = \frac{k}{q}.$
Therefore, $p = l$, $q = \frac{ad}{b}$, and both $\frac{d}{b}$ and $\frac{n}{l}$ are positive integers. Moreover, $q = \frac{ad}{b}$ and $\frac{n}{l}$ are divides $k$, this completes the proof.
\end{proof}
\noindent Based on the above analysis, we now investigate the existence of solution of the tensor-matrix
Eq\eqref{eq:4.1}. By exploiting the structure of the STP,
Eq\eqref{eq:4.1} can be equivalently transformed into the given form below:
\[
\mathcal{A}\ltimes X=
\bigl[
\widehat{\mathcal{A}}_{1}\;
\widehat{\mathcal{A}}_{2}\;
\cdots\;
\widehat{\mathcal{A}}_{p}
\bigr]
*
\bigl[
X_{1}\otimes \mathcal{I}_{\frac{n}{l}\times\frac{n}{l}\times r}\;
X_{2}\otimes \mathcal{I}_{\frac{n}{l}\times\frac{n}{l}\times r}\;
\cdots\;
X_{q}\otimes \mathcal{I}_{\frac{n}{l}\times\frac{n}{l}\times r}
\bigr]
=
\bigl[
\widehat{\mathcal{B}}_{1}\;
\widehat{\mathcal{B}}_{2}\;
\cdots\;
\widehat{\mathcal{B}}_{q}
\bigr].
\]
Furthermore, Eq\eqref{eq:4.1} can also be rewritten as
\begin{equation}\label{eq:4.2}
X \ltimes \mathcal{C}
=\bigl[
\overline{X}_1\;
\overline{X}_2\;
\cdots\;
\overline{X}_a
\bigr]
*\begin{bmatrix}
\widehat{\mathcal{C}}_1 &
\widehat{\mathcal{C}}_2 &
\cdots &
\widehat{\mathcal{C}}_a
\end{bmatrix}^{T}
=\sum_{g=1}^{a}
\overline{X}_g * \widehat{\mathcal{C}}_g
=\mathcal{D}.
\end{equation}
Here, $\widehat{\mathcal{A}}_j$ and $\widehat{\mathcal{B}}_j$ denote block tensors
of $\mathcal{A}$ and $\mathcal{B}$, respectively. Moreover,
\[\overline{X}_g
=\bigl[
X_{(g-1)\frac{d}{b}+1}\;
X_{(g-1)\frac{d}{b}+2}\;
\cdots\;
X_{g\frac{d}{b}}
\bigr]
\otimes I_{1\times1\times r},
\qquad
\widehat{\mathcal{C}}_g
=\mathcal{C}_g \otimes \mathcal{I}_{\frac{d}{b}\times \frac{d}{b}},\]
where $X_j$ denotes the $j^{th}$ column of $X$, and $\mathcal{C}_j$ represents the
$j^{th}$ horizontal slice of tensor $\mathcal{C}$. The dimensions satisfy
$\widehat{\mathcal{A}}_j,\, \widehat{\mathcal{B}}_j \in
\mathbb{C}^{m\times \frac{n}{l}\times r},\
\overline{X}_g \in \mathbb{C}^{p\times \frac{d}{b}},$
with $g=1,2,\ldots,a$ and $j=1,2,\ldots,q$. As a result, the following conclusions can be obtained.
\begin{theorem}
Eq\eqref{eq:4.1} admits a solution if and only if there exists a matrix $X \in \mathbb{C}^{p\times q}$ such that the resulting two conditions hold:\vspace{-0.5em}
\begin{enumerate}
\item[(i)] For each $j = 1,\ldots,q$, the collection of tensors
$\widehat{\mathcal{A}}_1, \widehat{\mathcal{A}}_2, \ldots, \widehat{\mathcal{A}}_p$
together with $\widehat{\mathcal{B}}_j$ is linearly dependent.\vspace{-0.5em}
\item[(ii)] The tensors
$\widehat{\mathcal{C}}_1, \widehat{\mathcal{C}}_2, \ldots, \widehat{\mathcal{C}}_a$
and $\mathcal{D}$ are linearly dependent.
\end{enumerate}\vspace{-0.5em}
Moreover, in order to explicitly construct solutions of the tensor
Eq\eqref{eq:4.1}, Eq\eqref{eq:4.2} can be further converted into a
vectorized form by applying the operator $\mathcal{V}_L(\cdot)$.
\end{theorem}\vspace{-0.5em}
Next, we provide example illustrating the above Lemma \ref{lemma:4.1}.\vspace{-0.5em}
\begin{example} Let $\mathcal{A,B,C}$ and $\mathcal{D}$ be $3^{rd}$-order tensors defined as
    \begin{center}
         \begin{tabular}{c c c c|c c c c}
\hline
  \multicolumn{4}{c}{$\mathcal{A}(:,:,1)$} & 
  \multicolumn{4}{c}{$\mathcal{A}(:,:,2)$} \\
   \hline 
   1&2&-1&0&0&3&-1&2\\
   3&1&2&-1&1&0&3&-1\\
   -1&3&1&2&0&1&0&3\\
   \hline
   \end{tabular},\quad
   \begin{tabular}{c c c c|c c c c}
\hline
  \multicolumn{4}{c}{$\mathcal{B}(:,:,1)$} & 
  \multicolumn{4}{c}{$\mathcal{B}(:,:,2)$} \\
   \hline 
   -2&2&-3&-2&-3&9&-2&1\\
   9&-2&1&-3&10&-3&5&-2\\
   2&9&3&1&0&10&0&5\\
   \hline
   \end{tabular},\quad
   \end{center}
   \begin{center}
   \begin{tabular}{c c c|c c c}
\hline
  \multicolumn{3}{c}{$\mathcal{C}(:,:,1)$} & 
  \multicolumn{3}{c}{$\mathcal{C}(:,:,2)$} \\
   \hline 
   1&2&0&3&0&1\\
\hline
\end{tabular}\quad
and \quad
\begin{tabular}{c c c c c c|c c c c c c}
\hline
\multicolumn{6}{c}{$\mathcal{D}(:,:,1)$} & 
\multicolumn{6}{c}{$\mathcal{D}(:,:,2)$} \\
\hline 
1&-1&2&-2&0&0&3&-3&0&0&1&-1\\
3&2&6&4&0&0&9&6&0&0&3&2\\
\hline
\end{tabular}. 
\end{center}
By Lemma \ref{lemma:4.1}, it is straightforward to verify that the given tensors are compatible.  A straight forward calculation will give the solution is
$X=\begin{bmatrix}
         1&-1\\
         3&2\\
     \end{bmatrix}.$
\end{example}
\subsection{The general case corresponds to \texorpdfstring{$m=h$}{m=h} with \texorpdfstring{$l\neq p$}{l=p}}
We study the existence of the solution of the tensor-matrix Eq\eqref{eq:4.1}
in the particular case where $m = h$ and $l = p$. Throughout the discussion, the tensors satisfy $\mathcal{A} \in \mathbb{C}^{m \times n \times r}$, $\mathcal{B} \in \mathbb{C}^{h \times k \times r}$, $\mathcal{C} \in \mathbb{C}^{a \times b \times r}$, and $\mathcal{D} \in \mathbb{C}^{l \times d \times r}$.
Adopting the methodology developed in Lemma~\ref{lemma:4.1}, we first establish the necessary conditions under which Eq\eqref{eq:4.1} admits a solution.
\begin{lemma}\label{lemma:4.3}
Let $X \in \mathbb{C}^{p\times q}$ be a solution to the tensor Eq\eqref{eq:4.1}. Then the dimensions of the tensors
$\mathcal{A}$, $\mathcal{B}$, $\mathcal{C}$, and $\mathcal{D}$ necessarily hold the following conditions:\vspace{-0.5em}
\begin{enumerate}
\item[(i)] The ratio $\frac{d}{b}$ is a positive integer; (ii) The dimensions obey
$p=\frac{l}{\beta}=\frac{n}{\alpha},\  q=\frac{ad}{b\beta}=\frac{k}{\alpha}$,
where $\alpha$ and $\beta\neq 1$ denote the greatest common divisors of $n$ with
$k$ and $l$ with $a$, respectively. In addition,
$\left[\beta,\frac{d}{b}\right]=1.$
\end{enumerate}
\end{lemma}
\begin{proof}
Let $X\in \mathbb{C}^{p\times q}$ satisfy Eq\eqref{eq:4.1} and from Eq\eqref{eq:4.1} we get,
$$\frac{mt_{1}}{n}=h, \frac{qt_{1}}{p}=k,\qquad \frac{pt_{2}}{q}=l, \frac{bt_{2}}{a}=d,$$
where $ t_{1}=[n, p] ,\ d_{1}=[r,1]=r,\  t_{2}=[q, a] \ \text{ and } \ d_{2}=[r,1]=r.$
 Then, for
because of $m=h$, we have
\[
\frac{t_1}{p}=\frac{n}{p}=\frac{k}{q}=\alpha.
\]
Moreover, $t_2=\frac{ad}{b}$, and \
$
\frac{t_2}{q}=\frac{ad}{bq}=\frac{l}{p}=\beta.
$
Since $\left(\beta,\tfrac{d}{b}\right)=1$, together with $\beta \mid a$ and
$\beta = 1$, it follows that $l = p$, which contradicts the assumption $l \neq p$. Therefore, $\tfrac{d}{b}$ must be a positive integer.
Moreover, the dimension relations are given by
\[
p=\frac{l}{\beta}=\frac{n}{\alpha},
\qquad
q=\frac{ad}{\alpha\beta}=\frac{k}{\beta}.
\]
This concludes the proof.
\end{proof}
\begin{remark}\label{remark:1}
Let $\beta_i \neq 1$, $i=1,\ldots,t$, be common divisors of $n$ and $k$, and let $\alpha_j$, $j=1,\ldots,s$, be common divisors of $\tfrac{ad}{b}$ and $l$.
Then Eq\eqref{eq:4.1} may admit solutions of dimensions
$ p_r \times q_r, \
p_r=\frac{n}{\alpha_r}=\frac{l}{\beta_r}, \
q_r=\frac{k}{\alpha_r}=\frac{ad}{b\beta_r}.$
Such dimensions are referred to as \emph{permissible sizes}. Moreover, solutions
corresponding to different permissible sizes are related as follows:\vspace{-0.5em}
\begin{itemize}
\item[(i)] Suppose that $X^{p_1 \times q_1}$ and $X^{p_2 \times q_2}$ are solutions
of two distinct permissible sizes satisfying
\[
\frac{q_2}{q_1}=\frac{p_2}{p_1}\in \mathbb{Z}.
\]
Then $X^{p_2 \times q_2}
=X^{p_1 \times q_1}\otimes I_{\frac{q_2}{q_1}\times \frac{q_2}{q_1}}$
, if Eq\eqref{eq:4.1} admits a unique solution of size
$p_2 \times q_2$, then any solution of size $p_1 \times q_1$, whenever it exists,
must also be unique.
\item[(ii)] Let $\widetilde{\alpha}=(n,k)$ and
$\widetilde{\beta}=\left[l,\tfrac{ad}{b}\right]$. Define
\[
\widetilde{p}=\frac{n}{\widetilde{\alpha}}=\frac{l}{\widetilde{\beta}}, \qquad
\widetilde{q}=\frac{k}{\widetilde{\alpha}}=\frac{ad}{b\widetilde{\beta}}.
\]
If Eq\eqref{eq:4.1} admits a solution of the minimal size
$\widetilde{p}\times\widetilde{q}$, then it also admits solutions for all permissible sizes
with $\beta\neq 1$.
\end{itemize}
\end{remark}\vspace{-0.5em}
Next, we provide example illustrating the above Lemma \ref{lemma:4.3}.
    \begin{example} Let $\mathcal{A,B,C}$ and $\mathcal{D}$ be $3^{rd}$ order tensors defined as
    \begin{center}
         \begin{tabular}{c c c c|c c c c}
\hline
  \multicolumn{4}{c}{$\mathcal{A}(:,:,1)$} & 
  \multicolumn{4}{c}{$\mathcal{A}(:,:,2)$} \\
   \hline 
   1&0&-1&1&2&-1&1&3\\
   0&1&0&-1&0&2&-1&1\\

   \hline
   \end{tabular},\quad
   \begin{tabular}{c c c c c c c c|c c c c c c c c}
\hline
  \multicolumn{8}{c}{$\mathcal{B}(:,:,1)$} & 
  \multicolumn{8}{c}{$\mathcal{B}(:,:,2)$} \\
   \hline 
   2&1&3&-1&-2&2&-3&1&7&0&3&-5&2&6&-3&5\\
   0&2&0&3&0&-2&0&-3&-1&7&1&3&-2&2&-1&-3\\
   \hline
\end{tabular},\quad
\end{center}
\begin{center}
\begin{tabular}{c|c}
\hline
\multicolumn{1}{c}{$\mathcal{C}(:,:,1)$} & 
\multicolumn{1}{c}{$\mathcal{C}(:,:,2)$} \\
   \hline 
   1&0\\[-.5em]
   0&1\\[-.5em]
   -1&-1\\[-.5em]
   2&0\\[-.5em]
   1&1\\[-.5em]
   0&1\\[-.1em]
   \hline
\end{tabular}\quad
and \quad
\begin{tabular}{c c|c c}
\hline
\multicolumn{2}{c}{$\mathcal{D}(:,:,1)$} & 
\multicolumn{2}{c}{$\mathcal{D}(:,:,2)$} \\
\hline 
3&-2&0&0\\[-.5em]
-2&3&-4&0\\[-.5em]
0&-2&3&-4\\[-.5em]
5&1&0&0\\[-.5em]
1&5&2&0\\[-.5em]
2&1&3&2\\[-.1em]
\hline
\end{tabular}. 
\end{center}
By Lemma \ref{lemma:4.3}, it is straightforward to verify that the given tensors are compatible.  A straight forward calculation will give the solution is
     $X=\begin{bmatrix}
         3&2&0&-2\\
         1&-1&2&1\\
     \end{bmatrix}.$
\end{example}
The following derivations will be used to prove some results.\\ We proceed to analyze the conditions of the existence of the solution of the tensor-matrix Eq\eqref{eq:4.1}. This section proceeds under the assumptions stated in Lemma~\ref{lemma:4.3}. Motivated by Remark~\ref{remark:1}, our analysis is confined to solutions of the minimal permissible size $p \times q$, since solutions corresponding to other permissible sizes can be obtained analogously. According to the definition of the STP, Eq\eqref{eq:4.1} admits the following equivalent representation
\begin{equation}
    \underbrace{\mathcal{A} \ltimes X}_{\frac{m t_1}{n} \times \frac{t_1}{p} \times d_1} =\mathcal{B},
\end{equation}
we have
$
\mathcal{A} \ltimes X=\bigl[\widehat{\mathcal{A}}_1 \ \widehat{\mathcal{A}}_2 \ \cdots \ \widehat{\mathcal{A}}_{p}]
*\bigl[X_1 \otimes \mathcal{I}_{\alpha \times \alpha \times r} \;\; X_2 \otimes \mathcal{I}_{\alpha \times \alpha \times r} \;\; \dots \;\; X_q \otimes \mathcal{I}_{\alpha \times \alpha \times r}\bigr]
=\bigl[\widehat{\mathcal{B}}_1 \ \widehat{\mathcal{B}}_2 \ \cdots \ \widehat{\mathcal{B}}_{\widetilde{q}}\bigr].$
And for
\begin{equation}\label{eq:4.8}
\underbrace{X \ltimes \mathcal{C}}_{pt_{2} \times \frac{bt_{2}}{a}\times d_2} = \mathcal{D},
\end{equation}

\noindent we begin by assuming that $X \in \mathbb{C}^{\widetilde{p} \times \widetilde{q}}$ is a
solution of Eq(4.6), where
\(
\widetilde{p}=\frac{l}{\alpha}
\quad \text{and} \quad
\widetilde{q}=\frac{ad}{b\beta}.
\)
Since $\frac{d}{b}$ and $\beta$ are coprime, let
\(
\widetilde{\beta}=l_3\cdot \frac{d}{b}+l_4, \, 
g=1,2,\ldots,p,\;
j=1,\ldots,b.
\)
Thus, we get
\begin{align*}
\operatorname{Row}_g(X)\ltimes \operatorname{Col}_j(\mathcal{C})
&=\bigl(\operatorname{Row}_g(X)\otimes I_{\widetilde{\beta} \times\widetilde{\beta}\times r }\bigr)
\bigl(\operatorname{Col}_j(\mathcal{C})\otimes I_{d\times d})\\
&=x_{g1}\mathcal{C}_{(::i)}^1+x_{g2}\mathcal{C}_{::i}^2+\cdots+x_{g,\frac{d}{b}}\mathcal{C}_{(::i)}^{\frac{d}{b}}+x_{g,\frac{d}{b}+1}\mathcal{C}_{(::i)}^{\frac{d}{b}+1}+\cdots+x_{g,q}\mathcal{C}_{(::i)}^{q}\\
&=\operatorname{Block}(\mathcal{D})\in \mathbb{C}^{\widetilde{\beta}\times\frac{d}{b}\times r}, \text{ which is a Toeplitz tensor.}
\end{align*}

Where
\begin{align*}
&{\mathcal{C}_{(::i)}^1}=
\begin{bmatrix}
c_{1,j,i} &  &  &  & \\
& \ddots &  &  & \\
&  & c_{1,j,i} &   & \\
 &  &  &\ddots & \\
 &    &  &  &c_{1,j,i}\\
 &  &\vdots &  & \\
c_{l_3 j} &  & & &     \\
& \ddots &  &  &    \\
&   & c_{l_3 j} &  &\\
&  &  &\ddots&\\
&  &   &   &c_{l_3,j,i}\\
c_{l_3+1,j} &   &   &  & \\
 & \ddots &  &  & \\
 &  & c_{l_3+1,j} & &\\
&   &  &\ddots&\\
&   &   &  &c_{l_{3} +1,j,i}
\end{bmatrix},
&{\mathcal{C}_{(::i)}^2}=
\begin{bmatrix}
    & &c_{l_3+1,j,i}& &\\
    & & &\ddots&\\
    & & & &c_{l_3+1,j,i}\\
    c_{l_3+2,j,i}& & & &\\
    &\ddots& & &\\
    & &c_{l_3+2,j,i}& &\\
    & & & \ddots&\\
    & & & &c_{l_3+2,j,i}\\
    & & \vdots& &\\
    c_{[\frac{2\widetilde{\beta}}{d/b}]+1,j,i}& & & &\\
    &\ddots& & &\\
    & & c_{[\frac{2\widetilde{\beta}}{d/b}]+1,j,i}& &\\
\end{bmatrix},
&\cdots\\
\end{align*}
\begin{align*}
&{\mathcal{C}_{(::i)}^{\frac{d}{b}}}=
\begin{bmatrix}
    & &c_{\widetilde{\beta}-l_3,j,i}& &\\
    & & &\ddots&\\
    & & & &c_{\widetilde{\beta}-l_3,j,i}\\
    c_{\widetilde{\beta}-l_3+1,j,i}& & & &\\
    &\ddots& & &\\
    & &c_{\widetilde{\beta}-l_3+1,j,i}& &\\
    & & & \ddots&\\
    & & & &c_{l_3+2,j,i}\\
    & & \vdots& &\\
    c_{\widetilde{\beta},j,i}& & & &\\
    &\ddots& & &\\
    & & c_{\widetilde{\beta},j,i}& &\\
    & & &\ddots&\\
    & & & & c_{\widetilde{\beta},j,i}\\
\end{bmatrix},
   & \quad{\mathcal{C}_{(::i)}^{\frac{d}{b}+1}}=
\begin{bmatrix}
      c_{\widetilde{\beta}+1,j,i}& & & & \\
     &\ddots& & &\\
    & &c_{\widetilde{\beta}+1,j,i}& &\\
    & & &\ddots&\\
    & & & &c_{\widetilde{\beta}+1,j,i}\\
    & &\vdots&&\\
    c_{\widetilde{\beta}+l_3,j,i}& & & &\\
    &\ddots& & &\\
    & &c_{\widetilde{\beta}+l_3,j,i}& &\\
    & & & \ddots&\\
    & & & &c_{\widetilde{\beta}+l_3,j,i}\\
   & & \vdots& &\\
 c_{\widetilde{\beta}+l_3+1,j,i}& & & &\\
 &\ddots& & &\\
 & &  c_{\widetilde{\beta}+l_3+1,j,i} & &
\end{bmatrix}.
&\cdots\\
\end{align*}
\begin{align*}
&{\mathcal{C}_{(::i)}^q}=
\begin{bmatrix}
    & &c_{a-l_3,j,i}& &\\
    & & &\ddots&\\
    & & & &c_{a-l_3,j,i}\\
    c_{a-l_3+1,j,i}& & & &\\
    &\ddots& & &\\
    & &c_{a-l_3+1,j,i}& &\\
    & & & \ddots&\\
    & & & &c_{a-l_3+1,j,i}\\
    & & \vdots& &\\
    c_{a,j,i}& & & &\\
    &\ddots& & &\\
    & & c_{a,j,i}& &\\
    &  &  &  \ddots&\\
    & & & &c_{a,j,i}
\end{bmatrix},
\end{align*}
using the procedure described above, we analyze the solvability conditions of Eq\eqref{eq:4.8}. Consequently, the following relations are derived.
\\
$x_{11}
\begin{bmatrix}
c_{1,1,i} & c_{1,2,i} & \cdots & c_{1,b,i}\\
c_{2,1,i} & c_{2,2,i} & \cdots & c_{2,b,i}\\
\vdots & \vdots & \ddots & \vdots\\
c_{l_3+1,1,i} & c_{l_3+1,2,i} & \cdots & c_{l_3+1,b,i}
\end{bmatrix}
+ x_{1,\frac{d}{b}+1}
\begin{bmatrix}
c_{\widetilde{\beta}+1,1,i} & c_{\widetilde{\beta}+1,2,i} & \cdots & c_{\widetilde{\beta}+1,b,i}\\
c_{\widetilde{\beta}+2,1,i} & c_{\widetilde{\beta}+2,2,i} & \cdots & c_{\widetilde{\beta}+2,b,i}\\
\vdots & \vdots & \ddots & \vdots\\
c_{\widetilde{\beta}+l_3+1,1,i}&c_{\widetilde{\beta}+l_3+1,2,i}&\cdots&c_{\widetilde{\beta}+l_3+1,b,i}
\end{bmatrix}\quad + \cdots$ \vspace{-0.8em}
\begin{align*}
& + x_{1,q-\frac{d}{b}+1}
\begin{bmatrix}
c_{a-\widetilde{\beta}+1,1,i} & c_{a-\widetilde{\beta}+1,2,i} & \cdots & c_{a-\widetilde{\beta}+1,b,i}\\
c_{a-\widetilde{\beta}+2,1,i} & c_{a-\widetilde{\beta}+2,2,i} & \cdots & c_{a-\widetilde{\beta}+2,b,i}\\
\vdots & \vdots & \ddots & \vdots\\
c_{a-\widetilde{\beta}+l_3+1,1,i} & c_{a-\widetilde{\beta}+l_3+1,2,i} & \cdots & c_{a-\widetilde{\beta}+l_3+1,b,i}
\end{bmatrix} 
=\begin{bmatrix}
d_{1,1,i} & d_{1,2,i} & \cdots & d_{1,d-\frac{d}{b},i}\\
d_{\frac{d}{b}+1,1,i} & d_{\frac{d}{b}+1,2,i} & \cdots & d_{\frac{d}{b}+1,d-\frac{d}{b},i}\\
\vdots & \vdots & \ddots & \vdots\\
d_{(l_3-1)\frac{d}{b}+1,1,i} & d_{(l_3-1)\frac{d}{b}+1,2,i} & \cdots & d_{(l_3-1)\frac{d}{b}+1,d-\frac{d}{b},i}\\
\end{bmatrix},
\end{align*}
\begin{align*}
&x_{12}
\begin{bmatrix}
c_{l_3+1,1,i} & c_{l_3+1,2,i} & \cdots & c_{l_3+1,b,i}\\
c_{l_3+2,1,i} & c_{l_3+2,2,i} & \cdots & c_{l_3+2,b,i}\\
\vdots & \vdots & \ddots & \vdots\\
c_{[\frac{2\widetilde{\beta}}{d/b}]+1,1,i} & c_{[\frac{2\widetilde{\beta}}{d/b}]+1,2,i} & \cdots & c_{[\frac{2\widetilde{\beta}}{d/b}]+1,b,i}
\end{bmatrix}
+ x_{1,\frac{d}{b}+2}
\begin{bmatrix}
c_{\widetilde{\beta}+l_3+1,1,i} & c_{\widetilde{\beta}+l_3+1,2,i} & \cdots & c_{\widetilde{\beta}+l_3+1,b,i}\\
c_{\widetilde{\beta}+l_3+2,1,i} & c_{\widetilde{\beta}+l_3+2,2,i} & \cdots & c_{\widetilde{\beta}+l_3+2,b,i}\\
\vdots & \vdots & \ddots & \vdots\\
c_{[\frac{2\widetilde{\beta}}{d/b}]+\widetilde{\beta}+1,1,i}&c_{[\frac{2\widetilde{\beta}}{d/b}]+\widetilde{\beta}+1,2,i}&\cdots&c_{[\frac{2\widetilde{\beta}}{d/b}]+\widetilde{\beta}+1,b,i}
\end{bmatrix}\quad + \cdots \\
&+ x_{1,q-\frac{d}{b}+2}
\begin{bmatrix}
c_{a-\widetilde{\beta}+l_3+1,1,i} & c_{a-\widetilde{\beta}+l_3+1,2,i} & \cdots & c_{a-\widetilde{\beta}+l_3+1,b,i}\\
c_{a-\widetilde{\beta}+l_3+2,1,i} & c_{a-\widetilde{\beta}+l_3+2,2,i} & \cdots & c_{a-\widetilde{\beta}+l_3+2,b,i}\\
\vdots & \vdots & \ddots & \vdots\\
c_{a-\widetilde{\beta}+[\frac{2\widetilde{\beta}}{d/b}]+1,1,i} & c_{a-\widetilde{\beta}+[\frac{2\widetilde{\beta}}{d/b}]+1,2,i} & \cdots & c_{a-\widetilde{\beta}+[\frac{2\widetilde{\beta}}{d/b}]+1,b,i}
\end{bmatrix}\\
&=\begin{bmatrix}
d_{1,l_4+1,i} & d_{1,\frac{d}{b}+l_4+1,i} & \cdots & d_{1,d-\frac{d}{b}+l_4+1,i}\\
d_{\frac{d}{b}+1,1,i} & d_{\frac{d}{b}-l_4+1,\frac{d}{b}+1,i} & \cdots & d_{\frac{d}{b}-l_4+1,d-\frac{d}{b},i}\\
\vdots & \vdots & \ddots & \vdots\\
d_{([\frac{2\widetilde{\beta}}{d/b}]-l_3)\frac{d}{b}-l_4+1,1,i} & d_{([\frac{2\widetilde{\beta}}{d/b}]-l_3)\frac{d}{b}-l_4+1,\frac{d}{b}+1,i} & \cdots & d_{([\frac{2\widetilde{\beta}}{d/b}]-l_3)\frac{d}{b}-l_4+1,d-\frac{d}{b},i}\\
\end{bmatrix},\\
&\vdots\\
&x_{1,\frac{d}{b}}
\begin{bmatrix}
c_{\widetilde{\beta}-l_3,1,i} & c_{\widetilde{\beta}-l_3,2,i} & \cdots & c_{\widetilde{\beta}-l_3,b,i}\\
c_{\widetilde{\beta}-l_3+1,1,i} & c_{\widetilde{\beta}-l_3+1,2,i} & \cdots & c_{\widetilde{\beta}-l_3+1,b,i}\\
\vdots & \vdots & \ddots & \vdots\\
c_{\widetilde{\beta},1,i} & c_{\widetilde{\beta},2,i} & \cdots & c_{\widetilde{\beta},b,i}
\end{bmatrix}
+ x_{1,\frac{2d}{b}}
\begin{bmatrix}
c_{2\widetilde{\beta}-l_3,1,i} & c_{2\widetilde{\beta}-l_3,2,i} & \cdots & c_{2\widetilde{\beta}-l_3,b,i}\\
c_{2\widetilde{\beta}-l_3+1,1,i} & c_{2\widetilde{\beta}+l_3+1,2,i} & \cdots & c_{2\widetilde{\beta}-l_3+1,b,i}\\
\vdots & \vdots & \ddots & \vdots\\
c_{2\widetilde{\beta},1,i}&c_{2\widetilde{\beta},2,i}&\cdots&c_{2\widetilde{\beta},b,i}
\end{bmatrix} \\
&\quad + \cdots + x_{1,q}
\begin{bmatrix}
c_{a-l_3,1,i} & c_{a-l_3,2,i} & \cdots & c_{a-l_3,b,i}\\
c_{a-l_3+1,1,i} & c_{a-l_3+1,2,i} & \cdots & c_{a-l_3+1,b,i}\\
\vdots & \vdots & \ddots & \vdots\\
c_{a,1,i} & c_{a,2,i} & \cdots & c_{a,b,i}
\end{bmatrix}\\
&=\begin{bmatrix}
d_{1,\frac{d}{b}-l_4,i} & d_{1,\frac{2d}{b}-l_4+1,i} & \cdots & d_{1,d-l_4+1,i}\\
d_{l_4+1,1,i} & d_{l_4+1,\frac{d}{b}+1,i} & \cdots & d_{l_4+1,d-\frac{d}{b}+1,i}\\
\vdots & \vdots & \ddots & \vdots\\
d_{(l_3-2)\frac{d}{b}+l_4+1,1,i} & d_{(l_3-2)\frac{d}{b}+l_4+1,\frac{d}{b}+1,i} & \cdots & d_{(l_3-2)\frac{d}{b}+l_4+1,d-\frac{d}{b}+1,i}\\
d_{\widetilde{\beta}-\frac{d}{b}+1,1,i}&d_{\widetilde{\beta}-\frac{d}{b}+1,\frac{d}{b}+1,i}&\cdots&d_{\widetilde{\beta}-\frac{d}{b}+1,d-\frac{d}{b}+1,i}
\end{bmatrix},\\
&x_{2,1}
\begin{bmatrix}
c_{1,1,i} & c_{1,2,i} & \cdots & c_{1,b,i} \\
c_{2,1,i} & c_{2,2,i} & \cdots & c_{2,b,i} \\
\vdots  & \vdots  & \ddots & \vdots  \\
c_{l_3+1,1,i} & c_{l_3+1,2,i} & \cdots & c_{l_3+1,b,i}
\end{bmatrix}
+ x_{2,\frac{d}{b}+1}
\begin{bmatrix}
c_{\widetilde{\beta}+1,1,i} & c_{\widetilde{\beta}+1,2,i} & \cdots & c_{\widetilde{\beta}+1,b,i} \\
c_{\widetilde{\beta}+2,1,i} & c_{\widetilde{\beta}+2,2,i} & \cdots & c_{\widetilde{\beta}+2,b,i} \\
\vdots & \vdots & \ddots & \vdots \\
c_{\widetilde{\beta}+l_3+1,1,i} & c_{\widetilde{\beta}+l_3+1,2,i} & \cdots & c_{\widetilde{\beta}+l_3+1,b,i}
\end{bmatrix} \\
&\qquad + \cdots + x_{2,q-\frac{d}{b}+2}
\begin{bmatrix}
c_{a-\widetilde{\beta}+1,1,i} & c_{a-\widetilde{\beta}+1,2,i} & \cdots & c_{a-\widetilde{\beta}+1,b,i} \\
c_{a-\widetilde{\beta}+2,1,i} & c_{a-\widetilde{\beta}+2,2,i} & \cdots & c_{a-\widetilde{\beta}+2,b,i} \\
\vdots & \vdots & \ddots & \vdots \\
c_{a-\widetilde{\beta}+l_3+1,1,i}& c_{a-\widetilde{\beta}+l_3+1,2,i} & \cdots & c_{a-\widetilde{\beta}+l_3+1,b,i}
\end{bmatrix} \\
&=\begin{bmatrix}
d_{\widetilde{\beta}+1,1,i} & d_{\widetilde{\beta}+1,2,i} & \cdots & d_{\widetilde{\beta}+1,d-\widetilde{\beta}+1,i} \\
d_{\widetilde{\beta}+2,1,i} & d_{\widetilde{\beta}+2,2,i} & \cdots & d_{\widetilde{\beta}+2,d-\mathcal{\beta}+1,i}\\
\vdots & \vdots & \ddots & \vdots \\
d_{(l_3-1)\frac{d}{b}+\widetilde{\beta}+1,1,i}& d_{(l_3-1)\frac{d}{b}+\beta+1,2,i}& \cdots& d_{(l_3-1)\frac{d}{b}+\widetilde{\beta}+1,\widetilde{\beta}+1,i}
\end{bmatrix},\\
&x_{2,2}
\begin{bmatrix}
c_{l_3+1,1,i} & c_{l_3+1,2,i} & \cdots & c_{l_3+1,b,i} \\
c_{l_3+2,1,i} & c_{l_3+2,2,i} & \cdots & c_{l_3+2,b,i} \\
\vdots & \vdots & \ddots & \vdots \\
c_{\left\lceil\frac{2\widetilde{\beta}}{d/b}\right\rceil +1,1,i}& c_{\left\lceil\frac{2\widetilde{\beta}}{d/b}\right\rceil +1,2,i}&\cdots& c_{\left\lceil\frac{2\widetilde{\beta}}{d/b}\right\rceil +1,b,i}
\end{bmatrix}+x_{2,\frac{d}{b}+2}
\begin{bmatrix}
c_{\widetilde{\beta}+l_3+1,1,i}& c_{\widetilde{\beta}+l_3+1,2,i}& \cdots & c_{\widetilde{\beta}+l_3+1,b,i} \\
c_{\widetilde{\beta}+l_3+2,1,i}& c_{\widetilde{\beta}+l_3+2,2,i}& \cdots & c_{\widetilde{\beta}+l_3+2,b,i} \\
\vdots & \vdots & \ddots & \vdots \\
c_{\left\lceil\frac{2\widetilde{\beta}}{d/b}\right\rceil+\widetilde{\beta}+1,1,i}& c_{\left\lceil\frac{2\widetilde{\beta}}{d/b}\right\rceil+\widetilde{\beta}+1,2,i}& \cdots& c_{\left\lceil\frac{2\widetilde{\beta}}{d/b}\right\rceil+\widetilde{\beta}+1,b,i}
\end{bmatrix} \ + \cdots\\
\end{align*}
\begin{align*}
& + x_{2,q-\frac{d}{b}+2}
\begin{bmatrix}
c_{a-\widetilde{\beta}+l_3+1,1,i}& c_{a-\widetilde{\beta}+l_3+1,2,i} & \cdots & c_{a-\widetilde{\beta}+l_3+1,b} \\
c_{a-\widetilde{\beta}+l_3+2,1,i} & c_{a-\widetilde{\beta}+l_3+2,2,i}& \cdots & c_{a-\widetilde{\beta}+l_3+2,b,i}\\
\vdots & \vdots & \ddots & \vdots \\
c_{a-\widetilde{\beta}+\left\lceil\frac{2\widetilde{\beta}}{d/b}\right\rceil+1,1,i}& c_{a-\widetilde{\beta}+\left\lceil\frac{2\widetilde{\beta}}{d/b}\right\rceil+1,2,i}&\cdots& c_{a-\widetilde{\beta}+\left\lceil\frac{2\widetilde{\beta}}{d/b}\right\rceil+1,b,i}
\end{bmatrix}\\
&=\begin{bmatrix}
d_{\widetilde{\beta}+1,l_4+1} & d_{\widetilde{\beta}+1,l_4+2,i}& \cdots & d_{\widetilde{\beta}+1,d-\widetilde{\beta}+1,i} \\
d_{\widetilde{\beta}+\frac{d}{b}-l_4+1,1,i}& d_{\widetilde{\beta}+\frac{d}{b}-l_4+1,2,i}& \cdots& d_{\widetilde{\beta}+\frac{d}{b}-4+1,d-\widetilde{\beta}+1,i} \\
\vdots & \vdots & \ddots & \vdots \\
d_{\left(\left\lceil\frac{2\widetilde{\beta}}{d/b}\right\rceil-l_3\right)\frac{d}{b}+\widetilde{\beta}-l_4+1,1,i}&d_{\left(\left\lceil\frac{2\widetilde{\beta}}{d/b}\right\rceil-1\right)\frac{d}{b}+\widetilde{\beta}-l_4+1,2}& \cdots &d_{\left(\left\lceil\frac{2\widetilde{\beta}}{d/b}\right\rceil-1\right)\frac{d}{b}+\widetilde{\beta}-l_4+1,d-\frac{d}{b}+1}
\end{bmatrix},\\
&\vdots\\
&x_{2,\frac{d}{b}}
\begin{bmatrix}
c_{\widetilde{\beta}-l_3,1,i} & c_{\widetilde{\beta}-l_3,2}& \cdots & c_{\widetilde{\beta}-l_3,b}\\
c_{\widetilde{\beta}-l_3+1,1,,i} & c_{\widetilde{\beta}-l_3+1,2,i} & \cdots & c_{\widetilde{\beta}-l_3+1,b,i} \\
\vdots & \vdots & \ddots & \vdots \\
c_{\widetilde{\beta},1,i} & c_{\widetilde{\beta},2,i} & \cdots & c_{\widetilde{\beta},b,i}
\end{bmatrix}+x_{2,\frac{2d}{b}}
\begin{bmatrix}
c_{2\widetilde{\beta}-l_3,1,i} & c_{2\widetilde{\beta}-l_3,2,i} & \cdots & c_{2\widetilde{\beta}-l_3,b,i} \\
c_{2\widetilde{\beta}-l_3+1,1,i}& c_{2\widetilde{\beta}-l_3+1,2,i}& \cdots & c_{2\widetilde{\beta}-l_3+1,b,i}\\
\vdots & \vdots & \ddots & \vdots \\
c_{2\widetilde{\beta},1,i} & c_{2\widetilde{\beta},2,i}& \cdots & c_{2\widetilde{\beta},b,i}
\end{bmatrix}\\
+\cdots &+ x_{2,q}
\begin{bmatrix}
c_{a-l_3,1,i} & c_{a-l_3,2,i} & \cdots & c_{a-l_3,b,i} \\
c_{a-l_3+1,1,i}& c_{a-l_3+1,2,i} & \cdots & c_{a-l_3+1,b,i}\\
\vdots & \vdots & \ddots & \vdots \\
c_{a,1,i} & c_{a,2,i} & \cdots & c_{a,b,i}
\end{bmatrix}\\
&=\begin{bmatrix}
 d_{\widetilde{\beta}+1,\frac{d}{b}-l_4,i} & d_{\widetilde{\beta}+1,\frac{2d}{b}-l_4+1,i}& \cdots & d_{\widetilde{\beta}+1,d-l_4+1,i} \\
d_{\widetilde{\beta}+l_4+1,1,i}& d_{\widetilde{\beta}+l_4+1,\frac{d}{b}+1,i}& \cdots& d_{\widetilde{\beta}+l_4+1,d-\frac{d}{b}+1,i} \\
\vdots & \vdots & \ddots & \vdots \\
d_{\left(l_3-2\right)\frac{d}{b}+\widetilde{\beta}+l_4+1,1,i}&d_{\left(l_3-2\right)\frac{d}{b}+\widetilde{\beta}+l_4+1,\frac{d}{b}+1,i}& \cdots &d_{\left(l_3-2\right)\frac{d}{b}+\widetilde{\beta}+l_4+1,d-\frac{d}{b}+1,i}\\
d_{2\widetilde{\beta}-\frac{d}{b}+1,1,i}&d_{2\widetilde{\widetilde{\beta}}-\frac{d}{b}+1,\frac{d}{b}+1,i}&\cdots&d_{2\widetilde{\beta}-\frac{d}{b}+1,d-\frac{d}{b}+1,i}
\end{bmatrix}\\
&x_{p,1}
\begin{bmatrix}
c_{1,1,i} & c_{1,2,i} & \cdots & c_{1,b,i} \\
c_{2,1,i} & c_{2,2,i} & \cdots & c_{2,b,i} \\
\vdots & \vdots & \ddots & \vdots \\
c_{l_3+1,1} & c_{l_3+1,2} & \cdots & c_{l_3+1,b}
\end{bmatrix}+ x_{p,\frac{d}{b}+1}
\begin{bmatrix}
c_{\widetilde{\beta}+1,1,i} & c_{\widetilde{\beta}+1,2,i} & \cdots & c_{\widetilde{\beta}+1,b,i} \\
c_{\widetilde{\beta}+2,1,i} & c_{\widetilde{\beta}+2,2,i} & \cdots & c_{\widetilde{\beta}+2,b,i} \\
\vdots & \vdots & \ddots & \vdots \\
c_{\widetilde{\beta}+l_3+1,1,i} & c_{\widetilde{\beta}+l_3+1,2,i} & \cdots & c_{\widetilde{\beta}+l_3+1,b,i}
\end{bmatrix}\\
&\quad + \cdots + x_{p,q-\frac{d}{b}+1}
\begin{bmatrix}
c_{a-\widetilde{\beta}+1,1,i} & c_{a-\widetilde{\beta}+1,2,i} & \cdots & c_{a-\widetilde{\beta}+1,b,i} \\
c_{a-\widetilde{\beta}+2,1} & c_{a-\widetilde{\beta}+2,2,i} & \cdots & c_{a-\widetilde{\beta}+2,b,i} \\
\vdots & \vdots & \ddots & \vdots \\
c_{a-\widetilde{\beta}+l_3+1,1,i} & c_{a-\widetilde{\beta}+l_3+1,2,i} & \cdots & c_{a-\widetilde{\beta}+l_3+1,b,i}
\end{bmatrix}\\
&=\begin{bmatrix}
d_{l-\widetilde{\beta}+1,1,i} & d_{l-\widetilde{\beta}+1,\frac{d}{b}+1,i} & \cdots & d_{l-\widetilde{\beta}+1,b,i} \\
d_{l-\widetilde{\beta}+\frac{d}{b}+1,1,i} & d_{l-\widetilde{\beta}+\frac{d}{b}+1,\frac{d}{b}+1,i} & \cdots & d_{l-\widetilde{\beta}+\frac{d}{b}+1,d-\frac{d}{b}+1,i} \\
\vdots & \vdots & \ddots & \vdots \\
d_{l-l_4-\frac{d}{b}+1,1,i}&d_{l-l_4-\frac{d}{b}+1,\frac{d}{b}+1}&\cdots&d_{l-l_4-\frac{d}{b}+1,d-\frac{d}{b}+1}\\
d_{l-l_4+1,1,i} & d_{l-l_4+1,\frac{d}{b}+1,i} & \cdots & d_{l-l_4+1,d-\frac{d}{b}+1,i}
\end{bmatrix}\\
&x_{p,2}
\begin{bmatrix}
c_{l_3+1,1,i} & c_{l_3+1,2,i} & \cdots & c_{l_3+1,b,i} \\
c_{l_3+2,1,i} & c_{l_3+2,2,i} & \cdots & c_{l_3+2,b,i} \\
\vdots & \vdots & \ddots & \vdots \\
c_{\left[\frac{2\widetilde{\beta}}{d/b}\right]+1,1,i} & c_{\left[\frac{2\widetilde{\beta}}{d/b}\right]+1,2,i} & \cdots & c_{\left[\frac{2\widetilde{\beta}}{d/b}\right]+1,b,i}
\end{bmatrix}
+ x_{p,\frac{d}{b}+2}
\begin{bmatrix}
c_{\widetilde{\beta}+l_3+1,1,i} & c_{\widetilde{\beta}+l_3+1,2,i} & \cdots & c_{\widetilde{\beta}+l_3+1,b,i} \\
c_{\widetilde{\beta}+l_3+2,1,i} & c_{\widetilde{\beta}+l_3+2,2,i} & \cdots & c_{\widetilde{\beta}+l_3+1+2,b,i} \\
\vdots & \vdots & \ddots & \vdots \\
c_{\left[\frac{2\widetilde{\beta}}{d/b}\right]+\widetilde{\beta}+1,1,i} & c_{\left[\frac{2\widetilde{\beta}}{d/b}\right]+\widetilde{\beta}+1,2,i} & \cdots & c_{\left[\frac{2\widetilde{\beta}}{d/b}\right]+\widetilde{\beta}+1,b,i}
\end{bmatrix} + \cdots\\
\end{align*}
\begin{align*}
& + x_{p,q-\bar d+2}
\begin{bmatrix}
c_{a-\widetilde{\beta}+l_3+1,1,i} & c_{a-\widetilde{\beta}+l_3+1,2,i} & \cdots & c_{a-\widetilde{\beta}+l_3+1,b,i} \\
c_{a-\widetilde{\beta}+l_3+2,1,i} & c_{a-\widetilde{\beta}+l_3+2,2,i} & \cdots & c_{a-\widetilde{\beta}+l_3+2,b,i}\\
\vdots & \vdots & \ddots & \vdots \\
c_{a-\widetilde{\beta}+\left[\frac{2\widetilde{\beta}}{d/b}\right]+1,1,i} & c_{a-\widetilde{\beta}+\left[\frac{2\widetilde{\beta}}{d/b}\right]+1,2,i} & \cdots & c_{a-\widetilde{\beta}+\left[\frac{2\widetilde{\beta}}{d/b}\right]+1,b,i}
\end{bmatrix}\\
&=\begin{bmatrix}
d_{l-\widetilde{\beta}+1,l_4+1,i} & d_{l-\widetilde{\beta}+1,\frac{d}{b}+l_4+1,i} & \cdots & d_{l-\widetilde{\beta}+1,d-\frac{d}{b}+l_4+1,i} \\
d_{l-\widetilde{\beta}+\frac{d}{b}-l_4+1,1,i} & d_{l-\widetilde{\beta}+\frac{d}{b}-l_4+1,\frac{d}{b}+1,i} & \cdots & d_{l-\widetilde{\beta}+\frac{d}{b}-l_4+1,d-\frac{d}{b}+1,i} \\
\vdots & \vdots & \ddots & \vdots \\
d_{\left[\frac{2\widetilde{\beta}}{d/b}\right]+l-2\widetilde{\beta}+1,1,i} & d_{\left[\frac{2\widetilde{\beta}}{d/b}\right]+l-2\widetilde{\beta}+1,\frac{d}{b}+1,i} & \cdots & d_{\left[\frac{2\widetilde{\beta}}{d/b}\right]+l-2\widetilde{\beta}+1,d-\frac{d}{b}+1,i}
\end{bmatrix}\\
&\vdots\\
&x_{p,\frac{d}{b}}
\begin{bmatrix}
c_{\widetilde{\beta}-l_3,1,i} & c_{\widetilde{\beta}-l_3,2,i} & \cdots & c_{\widetilde{\beta}-l_3,b,i} \\
c_{\widetilde{\beta}-l_3+1,1i} & c_{\widetilde{\beta}-l_3+1,2,i} & \cdots & c_{\widetilde{\beta}-l_3+1,b,i} \\
\vdots & \vdots & \ddots & \vdots \\
c_{\widetilde{\beta},1,i} & c_{\widetilde{\beta},2,i} & \cdots & c_{\widetilde{\beta},b,i}
\end{bmatrix}
+x_{p,\frac{2d}{b}}
\begin{bmatrix}
c_{2\widetilde{\beta}-l_3,1,i} & c_{2\widetilde{\beta}-l_3,2,i} & \cdots & c_{2\widetilde{\beta}-l_3,b,i} \\
c_{2\widetilde{\beta}-l_3+1,1,i} & c_{2\widetilde{\beta}-l_3+1,2,i} & \cdots & c_{2\widetilde{\beta}-l_3+1,b,i} \\
\vdots & \vdots & \ddots & \vdots \\
c_{2\widetilde{\beta},1,i} & c_{2\widetilde{\beta},2,i} & \cdots & c_{2\widetilde{\beta},b,i}
\end{bmatrix}\\
&\quad + \cdots +
x_{p,q} \begin{bmatrix}
c_{a-l_3,1,i} & c_{a-l_3,2,i} & \cdots & c_{a-l_3,b,i} \\
c_{a-l_3+1,1,i} & c_{a-l_3+1,2,i} & \cdots & c_{a-l_3+1,b,i} \\
\vdots & \vdots & \ddots & \vdots \\
c_{a,1,i} & c_{a,2,i} & \cdots & c_{a,b,i}
\end{bmatrix}
=\begin{bmatrix}
d_{l-\widetilde{\beta}+1,\frac{d}{b}-l_4+1,i} & d_{l-\widetilde{\beta}+1,\frac{2d}{b}-l_4+1.i} &
\cdots & d_{l-\widetilde{\beta}+1,d-l_4+1,i} \\
d_{l-\widetilde{\beta}+l_4+1,1,i} &
d_{l-\widetilde{\beta}+l_4+1,\frac{d}{b}+1,i} &
\cdots & d_{l-\widetilde{\beta}+l_4+1,d-\frac{d}{b}+1,i} \\
\vdots & \vdots & \ddots & \vdots \\
d_{l-\frac{2d}{b}+1,1,i} &
d_{l-\frac{2d}{b}+1,\frac{d}{b}+1,i} &
\cdots & d_{l-\frac{2d}{b}+1,d-\frac{d}{b}+1,i}\\
d_{l-\frac{d}{b}+1,1,i}&d_{l-\frac{d}{b}+1,\frac{d}{b}+1,i}&\cdots&d_{l-\frac{d}{b}+1,\frac{d}{b}+1,i}
\end{bmatrix}.
\end{align*}
For notational convenience, let $j.\beta = l_3^{\,j}\frac{d}{b} + l_4^{\,j},\ j = 1,2,\ldots,\frac{d}{b}.$ With this representation, the expressions can be simplified as follows:\\
$\left\{
\begin{aligned}
&x_{11}
\begin{bmatrix}
\mathcal{C}_1 \\ \mathcal{C}_2 \\ \vdots \\ \mathcal{C}_{l_3^1+1}
\end{bmatrix}
+x_{1,\frac{d}{b}+1}
\begin{bmatrix}
\mathcal{C}_{\widetilde{\beta}+1} \\ \mathcal{C}_{\widetilde{\beta}+2} \\ \vdots \\ \mathcal{C}_{\widetilde{\beta}+l_3^1+1}
\end{bmatrix}
+\cdots+
x_{1,q-\frac{d}{b}+1}
\begin{bmatrix}
\mathbb{C}_{a-\widetilde{\beta}+1} \\ \mathcal{C}_{a-\widetilde{\beta}+2} \\ \vdots \\ \mathcal{C}_{a-\widetilde{\beta}+l_3^1+1}
\end{bmatrix}
=\begin{bmatrix}
\bar{\mathcal{D}}_1 \\ \bar{\mathcal{D}}_{\frac{d}{b}+1} \\ \vdots \\ \bar{\mathcal{D}}_{l_3^1\frac{d}{b}+1}
\end{bmatrix},\\
&x_{12}
\begin{bmatrix}
\mathcal{C}_{l_3^1+1} \\ \mathcal{C}_{l_3^1+2} \\ \vdots \\ c_{l_3^2+1}
\end{bmatrix}
+x_{1,\frac{d}{b}+2}
\begin{bmatrix}
\mathcal{C}_{\widetilde{\beta}+l_3^1+1} \\ \mathcal{C}_{\widetilde{\beta}+l_3^1+2} \\ \vdots \\ \mathcal{C}_{\widetilde{\beta}+l_3^2+1}
\end{bmatrix}
+\cdots+x_{1,q-\frac{d}{b}+2}
\begin{bmatrix}
\mathcal{C}_{a-\widetilde{\beta}+l_3^1+1} \\ \mathcal{C}_{a-\widetilde{\beta}+l_3^1+2} \\ \vdots \\ \mathcal{C}_{a-\widetilde{\beta}+l_3^2+1}
\end{bmatrix}
=\begin{bmatrix}
\bar{\mathcal{D}}_{\widetilde{\beta}+l_4^1} \\ \bar{\mathcal{D}}_{\frac{d}{b}-l_4^1+1} \\ \vdots \\ \bar{\mathcal{D}}_{(l_3^2-l_3^1)\frac{d}{b}-l_4^1+1}
\end{bmatrix},\\
&\qquad \vdots\\
&x_{1,\frac{d}{b}}
\begin{bmatrix}
\mathcal{C}_{\widetilde{\beta}-l_3^1} \\ \mathcal{C}_{\widetilde{\beta}-l_3^1+1} \\ \vdots \\ \mathcal{C}_{\widetilde{\beta}}
\end{bmatrix}
+x_{1,\frac{2d}{b}}
\begin{bmatrix}
\mathcal{\beta}_{2\widetilde{\beta}-l_3^1} \\ \mathcal{C}_{2\widetilde{\beta}-l_3^1+1} \\ \vdots \\ \mathcal{C}_{2\widetilde{\beta}}
\end{bmatrix}
+\cdots+x_{1,q}
\begin{bmatrix}
\mathcal{C}_{a-l_3^1} \\ \mathcal{C}_{a-l_3^1+1} \\ \vdots \\ \mathcal{C}_a
\end{bmatrix}
=\begin{bmatrix}
\bar{\mathcal{D}}_{\widetilde{\beta}+\frac{d}{b}-l_4^1} \\ \bar{\mathcal{D}}_{l_4^1+1} \\ \vdots \\ \bar{\mathcal{D}}_{\widetilde{\beta}-\frac{d}{b}+1}
\end{bmatrix},\\
&x_{21}
\begin{bmatrix}
\mathcal{C}_1 \\ \mathcal{C}_2 \\ \vdots \\ \mathcal{C}_{l_3^1+1}
\end{bmatrix}
+x_{2,\frac{d}{b}+1}
\begin{bmatrix}
c_{\widetilde{\beta}+1} \\ c_{\widetilde{\beta}+2} \\ \vdots \\ c_{\widetilde{\beta}+l_3^1+1}
\end{bmatrix}
+\cdots+x_{2,q-\frac{d}{b}+1}
\begin{bmatrix}
\mathcal{C}_{a-\widetilde{\beta}+1} \\ \mathcal{C}_{a-\widetilde{\beta}+2} \\ \vdots \\ \mathcal{C}_{a-\widetilde{\beta}+l_3^1+1}
\end{bmatrix}
=\begin{bmatrix}
\bar{D}_{\widetilde{\beta}+\frac{d}{b}} \\ \bar{D}_{\widetilde{\beta}+2\frac{d}{b}} \\ \vdots \\ \bar{D}_{\widetilde{\beta}+(l_3^1+1)\frac{d}{b}}
\end{bmatrix},\\
\end{aligned}\right.$\\
$\left\{
\begin{aligned}
&x_{22}\begin{bmatrix}
\mathcal{C}_{l_3^1+1} \\ \mathcal{C}_{l_3^1+2} \\ \vdots \\ \mathcal{C}_{l_3^2+1}
\end{bmatrix}
+x_{2,\frac{d}{b}+2}
\begin{bmatrix}
\mathcal{C}_{\widetilde{\beta}+l_3^1+1} \\ \mathcal{C}_{\widetilde{\beta}+l_3^1+2} \\ \vdots \\ \mathcal{C}_{\widetilde{\beta}+l_3^2+1}
\end{bmatrix}
+\cdots+x_{2,q-\frac{d}{b}+2}
\begin{bmatrix}
\mathcal{C}_{a-\widetilde{\beta}+l_3^1+1} \\ \mathcal{C}_{a-\widetilde{\beta}+l_3^1+2} \\ \vdots \\ \mathcal{C}_{a-\widetilde{\beta}+l_3^1+1}
\end{bmatrix}
=\begin{bmatrix}
\bar{\mathcal{D}}_{2\widetilde{\beta}+l_4^1+\frac{d}{b}-1} \\ \bar{\mathcal{D}}_{\widetilde{\beta}+2\frac{d}{b}-l_4^1} \\ \vdots \\ \bar{\mathcal{D}}_{\widetilde{\beta}+(\frac{d}{b}-l_3^2-l_3^1+1)\frac{d}{b}-l_4^1}
\end{bmatrix},\\
&\vdots\\
&x_{2,\frac{d}{b}}
\begin{bmatrix}
\mathcal{C}_{\widetilde{\beta}-l_3^1} \\
\mathcal{C}_{\widetilde{\beta}-l_3^1+1} \\
\vdots \\
c_{\widetilde{\beta}}
\end{bmatrix}
+x_{2,\frac{2d}{b}}
\begin{bmatrix}
\mathcal{C}_{2\widetilde{\beta}-l_3^1} \\
\mathcal{C}_{2\widetilde{\beta}-l_3^1+1} \\
\vdots \\
\mathcal{C}_{2\widetilde{\beta}}
\end{bmatrix}
+\cdots+
x_{1,q}
\begin{bmatrix}
\mathcal{C}_{a-l_3^1} \\
\mathcal{C}_{a-l_3^1+1} \\
\vdots \\
\mathcal{C}_a
\end{bmatrix}
=\begin{bmatrix}
\bar{\mathcal{D}}_{2\widetilde{\beta}+2\frac{d}{b}-l_4^1-1} \\
\bar{\mathcal{D}}_{\widetilde{\beta}+\frac{d}{b}+l_4^1} \\
\vdots \\
\bar{\mathcal{D}}_{2\widetilde{\beta}}
\end{bmatrix},\\
&\vdots\\
&x_{p,1}
\begin{bmatrix}
\mathcal{C}_1 \\
\mathcal{C}_2 \\
\vdots \\
\mathcal{C}_{l_3^1+1}
\end{bmatrix}
+\cdots+
x_{p,\frac{d}{b}+1}
\begin{bmatrix}
\mathcal{C}_{a-\widetilde{\beta}+1} \\
\mathcal{C}_{a-\widetilde{\beta}+2} \\
\vdots \\
\mathcal{C}_{a-\widetilde{\beta}+l_3^1+1}
\end{bmatrix}
=\begin{bmatrix}
\bar{\mathcal{D}}_{(p-1)(\widetilde{\beta}+\frac{d}{b}-1)+1} \\
\bar{\mathcal{D}}_{(p-1)(\widetilde{\beta}+\frac{d}{b}-1)+\frac{d}{b}+1} \\
\vdots \\
\bar{\mathcal{D}}_{(p-1)(\widetilde{\beta}+\frac{d}{b}-1)+\frac{d}{b}+1}
\end{bmatrix},\\
&x_{p,2}
\begin{bmatrix}
\mathcal{C}_{l_3^1+1} \\
\mathcal{C}_{l_3^1+2} \\
\vdots \\
\mathcal{C}_{l_3^1+1}
\end{bmatrix}
+\cdots+
x_{p,p-\frac{d}{b}+2}
\begin{bmatrix}
\mathcal{C}_{a-\widetilde{\beta}+l_3^1+1} \\
\mathcal{C}_{a-\widetilde{\beta}+l_3^1+2} \\
\vdots \\
\mathcal{C}_{a-\widetilde{\beta}+l_3^2+1}
\end{bmatrix}
=\begin{bmatrix}
\bar{\mathcal{D}}_{(p-1)(\widetilde{\beta}+\frac{d}{b}-1)+\widetilde{\beta}+l_4} \\
\bar{\mathcal{D}}_{(p-1)(\widetilde{\beta}+\frac{d}{b}-1)+\widetilde{\beta}-l_4^1+1}\\
\vdots \\
\bar{\mathcal{D}}_{(p-1)(\widetilde{\beta}+\frac{d}{b}-1)+(l_3^2-l_3^1)\frac{d}{b}-l_4^1+1}
\end{bmatrix},\\
&\vdots\\
&x_{p,\frac{d}{b}}
\begin{bmatrix}
\mathcal{C}_{\widetilde{\beta}+l_3^1} \\
\mathcal{C}_{\widetilde{\beta}+l_3^1+1} \\
\vdots \\
\mathcal{C}_{\widetilde{\beta}}
\end{bmatrix}
+\cdots+
x_{p,q}
\begin{bmatrix}
\mathcal{C}_{a-l_3} \\ \mathcal{C}_{a-l_3^1+1} \\\vdots \\ \mathcal{C}_a
\end{bmatrix}
=\begin{bmatrix}
\bar{\mathcal{D}}_{(p-1)(\widetilde{\beta}+\frac{d}{b}-1)+\widetilde{\beta}+\frac{d}{b}-l_4^1} \\
\bar{\mathcal{D}}_{(p-1)(\widetilde{\beta}+\frac{d}{b}-1)+l_4^1+1} \\
\vdots \\ \bar{\mathcal{D}}_{(p-1)(\widetilde{\beta}+\frac{d}{b}-1)+\widetilde{\beta}-\frac{d}{b}+1}
\end{bmatrix}.
\end{aligned}
\right.$
\[
\bar{\mathcal{D}}=
\begin{bmatrix}
d_{1,1,i} & d_{1,\frac{d}{b}+1,i} & \cdots & d_{1,d-\frac{d}{b}+1,i} \\
d_{2,1,i} & d_{2,\frac{d}{b}+1,i} & \cdots & d_{2,d-\frac{d}{b}+1,i} \\
\vdots & \vdots & \ddots & \vdots \\
d_{\widetilde{\beta},1,i} & d_{\widetilde{\beta},\frac{d}{b}+1,i} & \cdots & d_{\widetilde{\beta},d-\frac{d}{b}+1,i} \\
d_{1,2,i} & d_{1,\frac{d}{b}+2,i} & \cdots & d_{1,d-\frac{d}{b}+2,i} \\
d_{1,3,i} & d_{1,\frac{d}{b}+3,i} & \cdots & d_{1,d-\frac{d}{b}+3,i} \\
\vdots & \vdots & \ddots & \vdots \\
d_{1,\frac{d}{b},i} & d_{1,2\frac{d}{b},i} & \cdots & d_{1,d,i} \\
\hdashline
\vdots & \vdots & \vdots & \vdots \\
\hdashline
d_{l-\widetilde{\beta}+1,1,i} & d_{l-\widetilde{\beta}+1,\frac{d}{b}+1,i} & \cdots &
d_{l-\widetilde{\beta}+1,d-\frac{d}{b}+1,i} \\
d_{l-\widetilde{\beta}+2,1,i} & d_{l-\widetilde{\beta}+2,\frac{d}{b}+1,i} & \cdots &
d_{l-\widetilde{\beta}+2,d-\frac{d}{b}+1,i} \\
\vdots & \vdots & \ddots & \vdots \\
d_{l,1,i} & d_{l,\frac{d}{b}+1,i} & \cdots & d_{l,d-\frac{d}{b}+1,i} \\
d_{l-\widetilde{\beta}+1,2,i} & d_{l-\widetilde{\beta}+1,\frac{d}{b}+2,i} & \cdots &
d_{l-\widetilde{\beta}+1,d-\frac{d}{b}+2,i} \\
\vdots & \vdots & \ddots & \vdots \\
d_{l-\widetilde{\beta}+1,\frac{d}{b},i} & d_{l-\widetilde{\beta}+1,2\frac{d}{b},i} & \cdots & d_{l-\widetilde{\beta}+1,d,i}
\end{bmatrix}.
\]
Let $\bar{\mathcal{D}}_t$ denote the $t^{th}$ horizontal slice of
$\bar{\mathcal{D}}$, and let $\mathcal{C}_r$ represent the $r^{th}$ horizontal
slice of $\mathcal{C}$. As a result, the following conclusions can be obtained.
From above discussion we have the following result.
\begin{theorem}\label{theorem4.4}
Tensor Eq\eqref{eq:4.1} admits a solution if and only if the following tensor–vector equations admit a solution
\begin{equation}
\begin{aligned}
&Y_{g,1} \ltimes
\bigl[
\check{\mathcal{C}}_1\;
\check{\mathcal{C}}_{\frac{d}{b}+1}\;
\cdots\;
\check{\mathcal{C}}_{(a/\widetilde{\beta}-1)\frac{d}{b}+1}
\bigr]^T
= \check{\mathcal{D}}_{g,1}, \\
&Y_{g,2} \ltimes
\bigl[
\check{\mathcal{C}}_2\;
\check{\mathcal{C}}_{\frac{d}{b}+2}\;
\cdots\;
\check{\mathcal{C}}_{(a/\widetilde{\beta}-1)\frac{d}{b}+2}
\bigr]^T
=\check{\mathcal{D}}_{g,2}, \\
&\vdots \\
&Y_{g,\frac{d}{b}} \ltimes
\bigl[
\check{\mathcal{C}}_{\frac{d}{b}}\;
\check{\mathcal{C}}_{\frac{2d}{b}}\;
\cdots\;
\check{\mathcal{C}}_q
\bigr]^T
=\check{\mathcal{D}}_{g,\frac{d}{b}}, \\
&(I_q \otimes \bar{\mathcal{A}})V_c(X) = V_c(B),
\end{aligned}
\end{equation}
where
$$\begin{aligned}
    &\mathcal{C} =\bigl[
\underbrace{\mathcal{C}_1\ \cdots\ \mathcal{C}_{l_3}}_{\check{\mathcal{C}}_1}\;
\mathcal{C}_{l_3+1}\;\cdots\;
\underbrace{\mathcal{C}_{a-l_3}\ \cdots\ \mathcal{C}_a}_{\check{C}_q}
\bigr]^T,\\
&\check{\mathcal{D}}_{g,1}
=\bigl[
\bar{\mathcal{D}}_{(g-1)(\widetilde{\beta}+\frac{d}{b}-1)+1}\;
\bar{\mathcal{D}}_{(g-1)(\widetilde{\beta}+\frac{d}{b}-1)+\frac{d}{b}+1}\;
\cdots\;
\bar{\mathcal{D}}_{(g-1)(\widetilde{\beta}+\frac{d}{b}-1)+l_3^1\frac{d}{b}+1}
\bigr]^T,\\
&\check{\mathcal{D}}_{g,2}
=\bigl[
\bar{\mathcal{D}}_{(g-1)(\widetilde{\beta}+\frac{d}{b}-1)+\widetilde{\beta}+l_4^1}\;
\bar{\mathcal{D}}_{(g-1)(\widetilde{\beta}+\frac{d}{b}-1)+\frac{d}{b}-l_4^1+1}\;
\cdots\;
\bar{\mathcal{D}}_{(g-1)(\widetilde{\beta}+\frac{d}{b}-1)+(l_3^2-l_3)\frac{d}{b}-l_4^1+1}
\bigr]^T,\\
&\vdots\\
&\check{\mathcal{D}}_{g,\frac{d}{b}}
=\bigl[
\bar{\mathcal{D}}_{(g-1)(\widetilde{\beta}+\frac{d}{b}-1)+\widetilde{\beta}+\frac{d}{b}-l_4}\;
\bar{\mathcal{D}}_{(g-1)(\widetilde{\beta}+\frac{d}{b}-1)+l_4+1}\;
\cdots\;
\bar{\mathcal{D}}_{(g-1)(\widetilde{\beta}+\frac{d}{b}-1)+\widetilde{\beta}+\frac{d}{b}+1}
\bigr]^T,
\end{aligned}$$
and vector $Y_{gj}=[y_{gj1}\ y_{gj2}\ ...\ y_{gj(a/\widetilde{\beta})}],$ then the matrix
$$\left[\begin{array}{*{13}c}
y_{111}&y_{121}&\cdots&y_{1(d/b)1}&y_{112}&y_{122}&\cdots&y_{1(d/b)2}&\cdots&y_{11(a/\widetilde{\beta})}&y_{12(a/\widetilde{\beta})}&\cdots&y_{1(d/b)(a/\widetilde{\beta})}  \\
y_{211}&y_{221}&\cdots&y_{2(d/b)1}&y_{212}&y_{222}&\cdots&y_{2(d/b)2}&\cdots&y_{21(a/\widetilde{\beta})}&y_{22(a/\widetilde{\beta})}&\cdots&y_{2(d/b)(a/\widetilde{\beta})}  \\
\vdots&\vdots&\vdots&\vdots&\vdots&\vdots&\vdots&\vdots&\vdots&\vdots&\vdots&\vdots&\vdots\\
y_{p11}&y_{p21}&\cdots&y_{p(d/b)1}&y_{p12}&y_{122}&\cdots&y_{p(d/b)2}&\cdots&y_{p1(a/\widetilde{\beta})}&y_{p2(a/\widetilde{\beta})}&\cdots&y_{p(d/b)(a/\widetilde{\beta})}  \\
\end{array}\right],
$$
$g=1\cdots p,\ j=1\cdots \frac{d}{b},$\\
$$\bar{\mathcal{A}}=[\mathcal{V}_L(\widehat{\mathcal{A}}_1)\ \mathcal{V}_L(\widehat{\mathcal{A}}_2)\ ...\ \mathcal{V}_L(\widehat{\mathcal{A}}_p)].$$
\end{theorem}
Next, we provide several example illustrating the above Theorem \ref{theorem4.4}.
\begin{example}
Take the tensors as follows:
\begin{center}
    \begin{tabular}{cccc|cccc}
     \hline
  \multicolumn{4}{c}{${\mathcal{A}}(:,:,1)$} & 
  \multicolumn{4}{c}{${\mathcal{A}}_{1}(:,:,2)$} \\
  \hline
         2&3&1&-1&0&1&1&-1  \\
         0&2&3&1&2&0&1&1 
    \end{tabular}, \quad
     \begin{tabular}{cccc|cccc}
     \hline
  \multicolumn{4}{c}{${\mathcal{B}}(:,:,1)$} & 
  \multicolumn{4}{c}{${\mathcal{B}}_{1}(:,:,2)$} \\
  \hline
         2&6&-4&-5&0&2&2&3  \\
         0&4&6&0&4&0&0&2 
    \end{tabular}, \quad
\end{center}
\begin{center}
        \begin{tabular}{cc|cc}
        \hline
  \multicolumn{2}{c}{${\mathcal{C}}(:,:,1)$} & 
  \multicolumn{2}{c}{${\mathcal{C}}(:,:,2)$} \\
  \hline
            1 &2&1&1  \\
             0&1&2&0\\
             3&0&0&1\\
             -1&-1&1&3
        \end{tabular}, \quad
        \begin{tabular}{cc|cc}
         \hline
  \multicolumn{2}{c}{${\mathcal{D}}(:,:,1)$} & 
  \multicolumn{2}{c}{${\mathcal{D}}(:,:,2)$} \\
  \hline
            -1 &4&2&1  \\
             1&3&3&-5\\
             6&0&0&2\\
             -2&-3&2&6
        \end{tabular}
    \end{center}
    The permissible size of solution is $2\times 2$ then
    \begin{center}
   \begin{tabular}{cc|cc}
      \hline
  \multicolumn{2}{c}{${\check{\mathcal{C}}_{1}(:,:,1)}$} & 
  \multicolumn{2}{c}{${\check{\mathcal{C}}_{1}(:,:,2)}$} \\
  \hline
         1&2&1&1  \\
         0&1&2&0 
    \end{tabular}, \quad
    \begin{tabular}{cc|cc}
     \hline
  \multicolumn{2}{c}{${\check{\mathcal{C}}_{2}(:,:,1)}$} & 
  \multicolumn{2}{c}{${\check{\mathcal{C}}_{2}(:,:,2)}$} \\
  \hline
    3&0&0&1  \\
         -1&-1&1&3 
    \end{tabular}, \quad
    \begin{tabular}{cc|cc}
     \hline
  \multicolumn{2}{c}{${\check{\mathcal{D}}_{11}(:,:,1)}$} & 
  \multicolumn{2}{c}{${\check{\mathcal{D}}_{11}(:,:,2)}$} \\
  \hline
    -1&4&2&1  \\
         1&3&3&-5 
    \end{tabular} , \quad 
     \begin{tabular}{cc|cc}
     \hline
  \multicolumn{2}{c}{${\check{\mathcal{D}}_{21}(:,:,1)}$} & 
  \multicolumn{2}{c}{${\check{\mathcal{D}}_{21}(:,:,2)}$} \\
  \hline
    6&0&0&2  \\
         -2&-2&2&10 
    \end{tabular}.
    \end{center}
    Hence, $Y_{11}=\begin{bmatrix}
        2 -1
    \end{bmatrix},\ Y_{21}=\begin{bmatrix}
        0 &2
    \end{bmatrix}$ and Eq\eqref{eq:4.1} admits a unique solution $X=\begin{bmatrix}
        2&-1\\
        0&2
    \end{bmatrix}$ .
    \end{example}
\subsection{The general case corresponds to \texorpdfstring{$m\neq h$}{m=h} with \texorpdfstring{$l=p$}{l=p}}
In this section, we examine the existence of solution of the tensor-matrix Eq\eqref{eq:4.1}, where the tensors $\mathcal{A,B,C}$ and $\mathcal{D}$ are invariant. Analogous to Lemma~\ref{lemma:4.3}, necessary conditions for the solvability of Eq\eqref{eq:4.1} are derived, and the dimensions of the solutions are also determined.
\begin{lemma}\label{lemma4.5}
Let $X \in \mathbb{C}^{p \times q}$ be a solution of the tensor Eq\eqref{eq:4.1}.  
Then the dimensions of the tensors $\mathcal{A}$, $\mathcal{B}$, $\mathcal{C}$, and $\mathcal{D}$ necessarily hold the following conditions:\vspace{-0.7em}
\begin{enumerate}
\item[(i)] The ratios $\frac{h}{m}$ and $\frac{d}{b}$ are positive integers;(ii) The dimensions satisfy
$p = l = \frac{n h}{\alpha m}, \ q = \frac{ad}{b} = \frac{k}{\alpha},$
where $\alpha$ is a common divisor of $n$ and $k$. Furthermore,
$(\alpha, \frac{h}{m}) = 1,$
and  $X$ is a unique solution.
\end{enumerate}
\end{lemma}
\begin{proof}
The proof follows the same arguments as those in Lemma~\ref{lemma:4.3} and is therefore omitted.
\end{proof}
We now analyze the solvability of Eq\eqref{eq:4.1}. The first step is to identify the permissible dimensions $p \times q$, which leads to
\[
\mathcal{A} \ltimes X = \mathcal{A} \ltimes [X_1, \dots, X_q] = [\mathcal{B}_1, \dots, \mathcal{B}_q],
\]
where $X_j$ denotes the $j^{th}$ column of $X$, and $\mathcal{B}_j \in
\mathbb{C}^{h \times \alpha \times r}$ is the corresponding block of
$\mathcal{B}$ for $j=1,\dots,q$. Consequently, solving
Eq\eqref{eq:4.1} reduces to solving $q$ tensor-vector equations using the STP, yielding the following results.
\begin{equation}\label{eq~4.6}
\mathcal{A} \ltimes X_j =
\begin{bmatrix}
\operatorname{Block}_{1j}(\mathcal{B}) \\
\operatorname{Block}_{2j}(\mathcal{B}) \\
\vdots \\
\operatorname{Block}_{mj}(\mathcal{B})
\end{bmatrix},
\quad j = 1,\ldots,q,
\end{equation}
where $\operatorname{Block}_{1j}(\mathcal{B}),\ldots,\operatorname{Block}_{mj}(\mathcal{B})$ are Toeplitz tensors.
Moreover, Eq\eqref{eq~4.6} can be equivalently rewritten as:
\[
X \ltimes \mathcal{C} =
\begin{bmatrix}
Y_1 &\cdots &Y_p
\end{bmatrix}^T\ltimes\mathcal{C}= 
\begin{bmatrix}
\mathcal{D}_1&\cdots&\mathcal{D}_p
\end{bmatrix}^T,
\]
Here, $Y_g$ denotes the $g^{th}$ row of $X$, and $\mathcal{D}_g \in \mathbb{C}^{1 \times d \times r}$ represents the corresponding block of $\mathcal{D}$. Then,we obtain
\begin{equation}
Y_g \ltimes \mathcal{C}
= (Y_g \otimes I_{\beta \times \beta \times r})(C \otimes I_{\frac{d}{b}\times \frac{d}{b}})
= \mathcal{D}_g ,\; g=1,...,p.
\end{equation}
Analogous to Theorem \ref{theorem:3.5}, we characterize the explicit structure of $\mathcal{B, \ D}$ using the properties of the STP and the dimensional compatibility condition.
\begin{theorem}\label{theorem:4.7}
Assume that Eq\eqref{eq:4.1} admits a solution $X \in \mathbb{C}^{p \times q}$.  
Then, the tensors $\mathcal{B}$ and $\mathcal{D}$ can be decomposed into sub-tensor, where each sub-tensor of $\mathcal{B}$ forms a Toeplitz tensor, and all blocks of $\mathcal{D}$ are identical.  
As a result, the forms of $\mathcal{B}$ and $\mathcal{D}$ are constrained as follows:
    \begin{equation}
\mathcal{B} =
\begin{bmatrix}
\operatorname{Block}_{11}(\mathcal{B}) & \cdots & \operatorname{Block}_{1q}(\mathcal{B}) \\
\vdots & \ddots & \vdots \\
\operatorname{Block}_{m1}(\mathcal{B}) & \cdots & \operatorname{Block}_{mq}(\mathcal{B})
\end{bmatrix},
\quad
\mathcal{D} =
\begin{bmatrix}
\mathcal{D}_1\\
\vdots\\
\mathcal{D}_p\\
\end{bmatrix}.
\end{equation}
Here, $\operatorname{Block}_{uj}(\mathcal{B}) \in \mathbb{C}^{\frac{h}{m} \times \frac{k}{q} \times r}$ denotes a Toeplitz tensor for $u = 1, \dots, m$ and $j = 1, \dots, q$, and $\mathcal{D}_g$ represents the $g^{th}$ horizontal slice of $\mathcal{D}$ for $g = 1, \dots, p$. Furthermore, the following condition is equivalent to the solvability of Eq\eqref{eq:4.1},
\begin{equation}
\left\{
\begin{aligned}
\mathcal{A} \ltimes X_1 &= \mathcal{B}_1 =
\begin{bmatrix}
\operatorname{Block}_{11}(\mathcal{B}) \\
\operatorname{Block}_{21}(\mathcal{B}) \\
\vdots \\
\operatorname{Block}_{m1}(\mathcal{B})
\end{bmatrix}
\in \operatorname{C}^{h \times \alpha \times r}, \\
&\vdots \\
\mathcal{A} \ltimes X_q &= \mathcal{B}_q =
\begin{bmatrix}
\operatorname{Block}_{1q}(\mathcal{B}) \\
\operatorname{Block}_{2q}(\mathcal{B}) \\
\vdots \\
\operatorname{Block}_{mq}(\mathcal{B})
\end{bmatrix}
\in \mathbb{C}^{h \times \alpha \times r}, \\[2mm]
Y_1 \ltimes \mathcal{C} &= \mathcal{D}_1 \in \mathbb{C}^{1 \times d \times r}, \\
&\vdots \\
Y_p \ltimes \mathcal{C} &= \mathcal{D}_p\in \mathbb{C}^{1 \times d \times r}.
\end{aligned}
\right.
\end{equation}
\end{theorem}

Next, we provide example illustrating the above Lemma \ref{lemma4.5}.\vspace{-0.5em}
\begin{example}
Let $\mathcal{A,B,C}$ and $\mathcal{D}$ be $3^{rd}$ order tensors defined as
\begin{center}
\begin{tabular}{c c c|c c c}
\hline
\multicolumn{3}{c}{$\mathcal{A}(:,:,1)$} & 
\multicolumn{3}{c}{$\mathcal{A}(:,:,2)$} \\
\hline 
1&0&-1&-1&1&2\\
2&1&0&0&-1&1\\
\hline
\end{tabular},\quad
\begin{tabular}{c c c c c c |c c c c c c}
\hline
\multicolumn{6}{c}{$\mathcal{B}(:,:,1)$} & 
\multicolumn{6}{c}{$\mathcal{B}(:,:,2)$} \\
\hline 
1&-2&-2&2&1&1&-3&5&6&-1&-1&-4\\
1&-2&1&2&1&3&-2&4&1&-1&-5&-4\\
-1&0&1&4&3&-3&4&-1&1&-3&2&-1\\
2&-3&3&4&3&2&3&-4&-4&-3&2&-5\\
\hline
\end{tabular},\quad
\end{center}
\begin{center}
\begin{tabular}{c|c}
\hline
\multicolumn{1}{c}{$\mathcal{C}(:,:,1)$} & 
\multicolumn{1}{c}{$\mathcal{C}(:,:,2)$} \\
\hline 
1&0\\
-1&1\\
\hline
\end{tabular}\quad
and \quad
\begin{tabular}{c c c|c c c}
\hline
\multicolumn{3}{c}{$\mathcal{D}(:,:,1)$} & 
\multicolumn{3}{c}{$\mathcal{D}(:,:,2)$} \\
\hline 
1&-3&1&1&2&-1\\
1&-5&0&1&2&0\\
-5&3&2&2&-1&-1\\
-4&4&1&2&-1&2\\
2&0&4&-1&1&-2\\
2&-2&2&-1&1&-3\\
\hline
\end{tabular}. 
\end{center}
By Lemma \ref{lemma4.5}, it is straightforward to verify that the given tensors are compatible. A straight forward calculation will give the solution is
$$X=\begin{bmatrix}
2&-1&0&1&2&-1\\
2&-3&0&1&2&0\\
-3&2&1&2&-1&-1\\
-2&3&3&2&-1&2\\
1&1&2&-1&1&-2\\
1&-1&-1&-1&1&-3\\
\end{bmatrix}.$$
\end{example}
\begin{remark} Assume that the tensor–matrix equation \eqref{eq:4.1} has a solution. If $\mathcal{A}$ is a F-diagonal tensor with identical diagonal entries, then each sub-tensor of $\mathcal{B}$ exhibits a Circulant tensor structure.
\end{remark}
\subsection{The general case corresponds to \texorpdfstring{$m\neq n$}{m=n} with \texorpdfstring{$l\neq p $}{l=p}}
In this section, we study the existence of the solution of the tensor Eq\eqref{eq:4.1} in the case $m\neq h$ and $p\neq l$, where the tensors $\mathcal{A,\ B,\ C}$ and $\mathrm{D}$ are given. A necessary condition for the existence of a solution follows directly from definition of STP.
\begin{lemma}\label{lemma:4.8}
Assume that the tensor Eq\eqref{eq:4.1} admits a solution $X \in \mathbb{C}^{p\times q}$. Then the dimensions of the tensors $\mathcal{A}$, $\mathcal{B}$, $\mathcal{C}$, and $\mathcal{D}$ necessarily hold the following conditions:\vspace{-0.7em}
\begin{enumerate}
\item[(i)] The ratios $\frac{h}{m}$ and $\frac{d}{b}$ are positive integers, (ii) The dimensions of $X$ are given by
$p=\frac{l}{\beta}=\frac{nh}{m\alpha},\
q=\frac{k}{\alpha}=\frac{ad}{b\beta},$
where $\alpha$ and $\beta \neq 1$ denote the greatest common divisors of $(n,k)$ and $(a,l)$, respectively. Furthermore,
$
(\alpha,\frac{h}{m})=1 
\quad \text{and} \quad
(\beta,\frac{d}{b})=1 .
$
\end{enumerate}
\end{lemma}
\begin{proof}
Let $X\in \mathbb{C}^{p\times q}$ satisfy Eq\eqref{eq:4.1} and from Eq\eqref{eq:4.1} we get,
$$\frac{mt_{1}}{n}=h, \frac{qt_{1}}{p}=k,\qquad \frac{pt_{2}}{q}=l, \frac{bt_{2}}{a}=d,$$
\noindent $\text{ where } t_{1}=[n, p], \ d_{1}=[r,1]=r,\ t_{2}=[q, a] \ \text{ and } \ d_{2}=[r,1]=r .$\\
Thus $\frac{t_{1}}{n}=\frac{h}{m}$ is a positive integer and $\frac{t_1}{p}=\frac{\frac{nh}{m}}{q}=\alpha$.
Then $\frac{t_2}{a}=\frac{d}{b}$ is positive integer and $\frac{t_2}{q}=\frac{l}{p}=\beta$, thus $p=\frac{l}{\beta}=\frac{nh}{\alpha m},\ q=\frac{k}{b}=\frac{ab}{b\beta}$.\\
Moreover, if $\beta=1$, then $l=p$, which does not satisfy the condition. Furthermore, we obtain
$$\frac{t}{p}=\alpha,\ \frac{t_1}{n}=\frac{h}{m};\ \frac{t_2}{a}=\beta,\ \frac{t_2}{a}=\frac{d}{b}.$$
Accordingly, $\left[\alpha,\frac{h}{m}\right]=1$ and $\left[\beta,\frac{d}{b}\right]=1$, this concludes the proof.
\end{proof}
\begin{remark}
If a solution of size $p \times q$ satisfies the dimensional requirements given in Lemma~\ref{lemma:4.8}, then this size is called \emph{permissible}. In analogy with Remark~\ref{remark:1}, assume that
$\frac{q_2}{p_1}=\frac{p_2}{q_1}>1,$
and let the solution of the tensor matrix  Eq\eqref{eq:4.1} be  $X \in \mathbb{C}^{p_1 \times q_1}$. Then
$X_1 = X \otimes I_{\frac{q_2}{p_1}\times \frac{q_2}{p_1}}$
is also a solution of Eq\eqref{eq:4.1}, where $X_1 \in \mathbb{C}^{p_2 \times q_2}$. Furthermore, if the solution $X_1$ of compatible size $p_2 \times q_2$ is unique, then the corresponding solution $X$ of compatible size $p_1 \times q_1$ is unique as well.
\end{remark}
Analogous to Theorem \ref{theorem:3.5}, we characterize the explicit structure of $\mathcal{B, \ D}$ using the properties of the STP and the dimensional compatibility condition.
\begin{theorem}\label{theomer:4.8}
If the tensor matrix Eq\eqref{eq:4.1} has a solution $X \in \mathbb{C}^{p \times q}$, then the tensors $\mathcal{B}$ and $\mathcal{D}$ must admit decompositions in which every block possesses a Toeplitz tensor. In particular, the tensors $\mathcal{B}$ and $\mathcal{D}$ can be expressed in the following block-wise forms:
\begin{equation}
\mathcal{B} =
\begin{bmatrix}
\operatorname{Block}_{11}(\mathcal{B}) & \cdots & \operatorname{Block}_{1q}(\mathcal{B}) \\
\vdots & \ddots & \vdots \\
\text{Block}_{m1}(\mathcal{B}) & \cdots & \text{Block}_{mq}(\mathcal{B})
\end{bmatrix},
\quad
\mathcal{D} =
\begin{bmatrix}
\text{Block}_{11}(\mathcal{D}) & \cdots & \text{Block}_{1b}(\mathcal{D}) \\
\vdots & \ddots & \vdots \\
\text{Block}_{p1}(\mathcal{D}) & \cdots & \text{Block}_{pb}(\mathcal{D})
\end{bmatrix}.
\end{equation}
where $\operatorname{Block}_{sj}(\mathcal{B}) \in \mathbb{C}^{\frac{h}{m} \times \frac{k}{q}\times r}$ and
$\operatorname{Block}_{gt}(\mathcal{D}) \in \mathbb{C}^{\beta \times \frac{d}{b}\times r}$ are Toeplitz tensor, with
$j = 1,\ldots,m$, $j = 1,\ldots,q,\ g=1,\ldots,p,\ t=1,\ldots,b$.
\end{theorem}
\begin{proof}
For Eq\eqref{eq:4.2}, we have
\[\mathcal{A} \ltimes X = \mathcal{A} \ltimes [X_1,\ldots,X_q] = [\mathcal{B}_1,\ \mathcal{B}_2,\ldots,\mathcal{B}_q],\]
where $X_j$ represents the $j^{th}$ column of $X$ and $\mathcal{B}_j \in \mathbb{C}^{h \times \alpha \times r}$ represents the corresponding block of $\mathcal{B}$ for $j=1,\ldots,q$. Consequently, the original problem reduces to solving $q$ tensor–vector equations formulated via the STP. By applying the same procedure as in Case~3.2, we obtain
\begin{equation}
\mathcal{A} \ltimes X_j =\mathcal{B}_j
\begin{bmatrix}
\text{Block}_{1j}(\mathcal{B}) \\
\text{Block}_{2j}(\mathcal{B}) \\
\vdots \\
\text{Block}_{mj}(\mathcal{B})
\end{bmatrix},
\quad j = 1,\ldots,q,
\end{equation}
where $\text{Block}_{1j}(\mathcal{B}),\ldots,\text{Block}_{mj}(\mathcal{B})$ are Toeplitz tensors.
Moreover, Eq(3.6) can be rewritten as
\[
X \ltimes \mathcal{C} =
\begin{bmatrix}
Y_1 &\cdots&Y_p\\
\end{bmatrix}^T\ltimes\mathcal{C},
\quad
\mathcal{C} = 
\begin{bmatrix}
\mathcal{D}_1 &\cdots &\mathcal{D}_p
\end{bmatrix}^T,
\]
where $Y_g$ is the $g^{th}$ row of $X$ and $\mathcal{D}_g \in \mathbb{C}^{a \times b \times r}$
is a block of $\mathcal{D}$, $g = 1,\ldots,p$. Then,we obtain
\begin{equation}
Y_g \ltimes \mathcal{C}
= (Y_g \otimes I_{\beta \times \beta \times r})(\mathcal{C} \otimes I_{\frac{d}{b}\times \frac{d}{b}})
= \mathcal{D}_g = [\text{Block}_{g1}(\mathcal{D}),\ldots,\text{Block}_{gb}(\mathcal{D})],
\end{equation}
where $\text{Block}_{gt}(\mathcal{D})$ is a Toeplitz tensor,
$g = 1,\ldots,p$, $t = 1,\ldots,b$.  This concludes the proof.
\end{proof}
The following result provides an equivalence characterization for the solvability of the Eq\eqref{eq:4.1}.
\begin{theorem}
The solvability of Eq\eqref{eq:4.1} is equivalent to the solvability of the following system:
\begin{equation}
\left\{
\begin{aligned}
\mathcal{A} \ltimes X_1 &= \mathcal{B}_1 =
\begin{bmatrix}
\mathrm{Block}_{11}(\mathcal{B}) \\
\mathrm{Block}_{21}(\mathcal{B}) \\
\vdots \\
\mathrm{Block}_{m1}(\mathcal{B})
\end{bmatrix}
\in \mathbb{C}^{h \times \alpha \times r}, \\[2mm]
&\vdots \\[2mm]
\mathcal{A} \ltimes X_q &= \mathcal{B}_q =
\begin{bmatrix}
\mathrm{Block}_{1q}(\mathcal{B}) \\
\mathrm{Block}_{2q}(\mathcal{B}) \\
\vdots \\
\mathrm{Block}_{mq}(\mathcal{B})
\end{bmatrix}
\in \mathbb{C}^{h \times \alpha \times r}, \\[2mm]
Y_1 \ltimes \mathcal{C} &= \mathcal{D}_1 =
[\mathrm{Block}_{11}(\mathcal{D}),\ldots,\mathrm{Block}_{1b}(\mathcal{D})]
\in \mathbb{C}^{\beta \times d \times r}, \\[2mm]
&\vdots \\[2mm]
Y_p \ltimes \mathcal{C} &= \mathcal{D}_p =
[\mathrm{Block}_{p1}(\mathcal{D}),\ldots,\mathrm{Block}_{pb}(\mathcal{D})]
\in \mathbb{C}^{\beta \times d \times r}.
\end{aligned}
\right.
\tag{4.14}
\end{equation}
where $X_j$ is the $j^{th}$ column of $X$, $\mathrm{Block}(\mathcal{B})_{uj} \in
\mathbb{C}^{\frac{h}{m} \times \alpha \times r}$ is a Toeplitz tensor; $Y_g$ is the $g^{th}$ row of $X$ and $\mathrm{Block}(D)_{gs} \in
\mathbb{C}^{\beta \times \frac{d}{b}\times r}$ is also a Toeplitz tensor, with
$u = 1,\ldots,m$, $j = 1,\ldots,q$, $s = 1,\ldots,b$, and $g = 1,\ldots,p$.
\end{theorem}
Next, we provide example illustrating the above Lemma \ref{lemma:4.8}.
\begin{example}
    Let $\mathcal{A,B,C}$ and $\mathcal{D}$ be $3^{rd}$ order tensors defined  as
\begin{center}
\begin{tabular}{c|c}
\hline
  \multicolumn{1}{c}{$\mathcal{A}(:,:,1)$} & 
  \multicolumn{1}{c}{$\mathcal{A}(:,:,2)$} \\
   \hline 
   2&-1\\
   1&1\\
 \hline
\end{tabular},\quad
   \begin{tabular}{c c c|c c c}
\hline
\multicolumn{3}{c}{$\mathcal{B}(:,:,1)$} & 
\multicolumn{3}{c}{$\mathcal{B}(:,:,2)$} \\
   \hline 
  4&-2&6&-2&1&-2\\
  2&6&0&-1&-3&0\\
  2&-1&2&2&-1&2\\
  1&3&0&1&3&0\\
   \hline
\end{tabular},\quad
\end{center}
\begin{center}
\begin{tabular}{c|c}
\hline
  \multicolumn{1}{c}{$\mathcal{C}(:,:,1)$} & 
  \multicolumn{1}{c}{$\mathcal{C}(:,:,2)$} \\
   \hline 
   3&-1\\
   2&1\\
   \hline
\end{tabular}\quad
and \quad
\begin{tabular}{c c c|c c c}
\hline
\multicolumn{3}{c}{$\mathcal{D}(:,:,1)$} & 
\multicolumn{3}{c}{$\mathcal{D}(:,:,2)$} \\
   \hline 
   6&4&-3&-2&2&1\\
   -2&6&4&-1&-2&2\\
   3&0&9&-1&0&-3\\
   6&3&0&3&-1&0\\
   \hline
\end{tabular}. 
\end{center}
By Lemma \ref{lemma:4.8}, it is straightforward to verify that the given tensors are compatible. The solution is
     $X=\begin{bmatrix}
         2&-1&2\\
         1&3&0\\
     \end{bmatrix}.$
\end{example}
\begin{remark} Assume that the tensor–matrix equation \eqref{eq:4.1} has a solution. If $\mathcal{A}$ and $\mathcal{C}$ are  F-diagonal tensors with identical diagonal entries, then each sub-tensor of $\mathcal{B,\ D}$ exhibits a Circulant tensor structure.
\end{remark}
\section{Solvability conditions for tensor equations with tensor solutions}
This section, we examine the existence of solution of the tensor tensor equation defined via the STP:
\begin{equation}\label{eq:5.1}
\left\{
\begin{aligned}
& \underbrace{\mathcal{A} \ltimes \mathcal{X}}_{\frac{m t_1}{n} \times \frac{t_1}{p} \times d_1} = \mathcal{B} \\
& \underbrace{\mathcal{X} \ltimes \mathcal{C}}_{pt_{2} \times \frac{bt{2}}{a}\times d_2} = \mathcal{D},
\end{aligned}
\right.
\end{equation}\\
where $\mathcal{A} \in \mathbb{C}^{m \times n \times r}$, $\mathcal{B} \in \mathbb{C}^{h \times k \times r}$,
$\mathcal{C} \in \mathbb{C}^{a \times b \times r}$, and $\mathcal{D} \in \mathbb{C}^{l \times b \times r}$ are know tensors, and $\mathcal{X} \in \mathbb{C}^{p \times q \times r}$ is the unknown tensor. Here,
$t_{1}=[n,p]$, $t_{2}=[q,a]$, and $d_{1}=d_{2}=[r,r]$. 
We begin with the special case $m=h$ and $l=p$, after which the analysis is extended to the general case.
\subsection{The simplified case with \texorpdfstring{$m=h$}{m=h} and \texorpdfstring{$l=p$}{l=p}}
The solvability of tensor Eq\eqref{eq:5.1} is considered in this section, and the following lemma is obtained via the STP.
\begin{lemma}\label{lemma:5.1}
Let $X \in \mathbb{C}^{p \times q\times r}$ be a solution of the tensor Eq\eqref{eq:5.1}. Then, the dimensions of the tensors $\mathcal{A}, \mathcal{B}, \mathcal{C}$, and $\mathcal{D}$ necessarily hold the following conditions:\vspace{-0.7em}
\begin{enumerate}
    \item[(i)] The ratios $\frac{n}{l}$ and $\frac{d}{b}$ are positive integers; (ii) The ratios $\frac{ad}{b}$ and $\frac{n}{l}$ divide $k$, and $q = \frac{ad}{b} = \frac{l k}{n}$.
\end{enumerate}
\end{lemma}
\begin{proof}
The proof follows the analogous arguments as those in Lemma~\ref{lemma:4.1} and is therefore omitted.
\end{proof}
Next, we provide example illustrating the above Lemma \ref{lemma:5.1}.\vspace{-0.5em}
\begin{example}
    Let $\mathcal{A,\ B,\ C}$ and $\mathcal{D}$ are tensors defined as
    \begin{center}
\begin{tabular}{c c c c|c c c c}
\hline
  \multicolumn{4}{c}{$\mathcal{A}(:,:,1)$} & 
  \multicolumn{4}{c}{$\mathcal{A}(:,:,2)$} \\
   \hline 
   1&2&-1&0&2&-1&1&2\\
   3&1&2&-1&1&2&-1&1\\
   2&3&1&2&0&1&2&-1\\
 \hline
\end{tabular},\quad
   \begin{tabular}{c c c c|c c c c}
\hline
\multicolumn{4}{c}{$\mathcal{B}(:,:,1)$} & 
\multicolumn{4}{c}{$\mathcal{B}(:,:,2)$} \\
   \hline 
  9&4&8&-5&9&12&10&13\\
  17&9&13&8&15&9&12&10\\
  11&14&2&13&16&15&15&12\\
   \hline
\end{tabular},\quad
\end{center}
\begin{center}
\begin{tabular}{c c c|c c c}
\hline
  \multicolumn{3}{c}{$\mathcal{C}(:,:,1)$} & 
  \multicolumn{3}{c}{$\mathcal{C}(:,:,2)$} \\
   \hline 
   2&1&2&-1&1&3\\
   \hline
\end{tabular}\quad
and\quad
\begin{tabular}{c c c c c c|c c c c c c}
\hline
\multicolumn{6}{c}{$\mathcal{D}(:,:,1)$} & 
\multicolumn{6}{c}{$\mathcal{D}(:,:,2)$} \\
   \hline 
   -2&-3&6&6&16&17&6&9&6&6&14&13\\
   4&5&5&1&12&1&1&-4&5&1&13&4\\
   \hline
\end{tabular}. 
\end{center}
By Lemma \ref{lemma:5.1}, it is straightforward to verify that the given tensors are compatible. The solution is $\mathcal{X}$
\begin{tabular}{c c|c c}
\hline
\multicolumn{2}{c}{$\mathcal{X}(:,:,1)$} & 
\multicolumn{2}{c}{$\mathcal{X}(:,:,2)$} \\
   \hline 
  2&1&4&5\\
  3&2&2&-1\\
   \hline
\end{tabular}.
\end{example}
\subsection{The general case corresponds to \texorpdfstring{$m=h$}{m=h} with \texorpdfstring{$l\neq p$}{l=p}}
In this section, we examine the existence of solution of the  tensor-tensor Eq\eqref{eq:5.1}
in the particular case where $m = h$ and $l = p$. Throughout the discussion, the tensors satisfy $\mathcal{A} \in \mathbb{C}^{m \times n \times r}$, $\mathcal{B} \in \mathbb{C}^{h \times k \times r}$, $\mathcal{C} \in \mathbb{C}^{a \times b \times r}$, and $\mathcal{D} \in \mathbb{C}^{l \times d \times r}$.
Adopting the methodology developed in Lemma~\ref{lemma:5.1}, we first establish the necessary conditions under which Eq\eqref{eq:5.1} admits a solution.
\begin{lemma}\label{lemma:5.2}
Let $X \in \mathbb{C}^{p\times q\times r}$ be a solution to the tensor tensor Eq\eqref{eq:5.1}. Then the dimensions of the tensors
$\mathcal{A}$, $\mathcal{B}$, $\mathcal{C}$, and $\mathcal{D}$ necessarily hold the following conditions:\vspace{-0.7em}
\begin{enumerate}
\item[(i)] The ratio $\frac{d}{b}$ is a positive integer; (ii) The dimensions obey
$p=\frac{l}{\beta} =\frac{n}{\alpha}, \
q=\frac{ad}{b\beta}=\frac{k}{\alpha}$,
where $\alpha$ and $\beta\neq 1$ denote the greatest common divisors of $n$ with
$k$ and $l$ with $b$, respectively. In addition,
$(\beta,\frac{d}{b})=1.$
\end{enumerate}
\end{lemma}
\begin{proof}
The proof follows the analogous arguments as those in Lemma \ref{lemma:4.3} and is therefore omitted.
\end{proof}
\begin{remark}\label{remark:5.1}
Suppose that $\alpha_i$, $i=1,\ldots,s$, are common divisors of $n$ and $k$, and that
$\beta_j \neq 1$, $j=1,\ldots,t$, are common divisors of $l$ and $\frac{ad}{b}$.
Then Eq\eqref{eq:5.1} may admit solutions whose dimensions satisfy
$
p_r \times q_r \times r, \
p_r=\frac{n}{\alpha_r}=\frac{l}{\beta_r}, \
q_r=\frac{k}{\alpha_r}=\frac{ad}{b\beta_r}.
$
Such dimensions are referred to as \emph{permissible sizes}. Moreover, solutions corresponding to different permissible sizes are related as follows:\vspace{-0.7em}
\begin{itemize}
\item[(i)] Let $X^{p_1 \times q_1 \times r}$ and $X^{p_2 \times q_2 \times r}$ be two solutions of permissible sizes satisfying
\[
\frac{q_2}{q_1}=\frac{p_2}{p_1}\in \mathbb{Z}, \qquad \frac{q_2}{q_1}>1.
\]
Then the larger solution can be constructed from the smaller one via
\[
X^{p_2 \times q_2 \times r}
=
X^{p_1 \times q_1 \times r} \otimes I_{\frac{q_2}{q_1}\times \frac{q_2}{q_1}}.
\]
Conversely, if Eq\eqref{eq:5.1} possesses a unique solution of permissible size
$p_2 \times q_2 \times r$, then any solution of size $p_1 \times q_1 \times r$, if it exists, must also be unique.
\vspace{-0.7em}
\item[(ii)] Define $\widetilde{\alpha}=\gcd\{n,k\}$ and $\widetilde{\beta}=\gcd\{l,\tfrac{ad}{b}\}$.
Let
\[
\widetilde{p}=\frac{n}{\widetilde{\alpha}}=\frac{l}{\widetilde{\beta}}, \qquad
\widetilde{q}=\frac{k}{\widetilde{\alpha}}=\frac{ad}{b\widetilde{\beta}}.
\]
If Eq~\eqref{eq:5.1} admits a solution of minimal dimension
$\widetilde{p} \times \widetilde{q} \times r$, then it also admits solutions of all permissible sizes for which $\beta \neq 1$.
\end{itemize}
\end{remark}\vspace{-0.7em}
Next, we provide example illustrating the above Lemma \ref{lemma:5.2}.\vspace{-0.6em}
\begin{example}
    Let $\mathcal{A,\ B,\ C}$ and $\mathcal{D}$ are tensors defined as
     \begin{center}
\begin{tabular}{c c c c|c c c c}
\hline
  \multicolumn{4}{c}{$\mathcal{A}(:,:,1)$} & 
  \multicolumn{4}{c}{$\mathcal{A}(:,:,2)$} \\
   \hline 
   1&3&0&2&5&0&6&1\\
   4&1&3&0&2&5&0&6\\
 \hline
\end{tabular},\quad
\begin{tabular}{c|c}
\hline
  \multicolumn{1}{c}{$\mathcal{C}(:,:,1)$} & 
  \multicolumn{1}{c}{$\mathcal{C}(:,:,2)$} \\
   \hline 
   1&3\\
   2&1\\
   3&0\\
   0&1\\
   1&2\\
   0&1\\
   \hline
\end{tabular},\quad
\quad
\begin{tabular}{c c|c c}
\hline
\multicolumn{2}{c}{$\mathcal{D}(:,:,1)$} & 
\multicolumn{2}{c}{$\mathcal{D}(:,:,2)$} \\
   \hline 
 11&15&8&12\\
 12&11&7&8\\
 9&12&12&7\\
 10&10&19&11\\
 7&10&8&19\\
 18&7&15&8\\
   \hline
\end{tabular},
\end{center}
\begin{center}
 and  \begin{tabular}{c c c c c c c c|c c c c c c c c}
\hline
\multicolumn{8}{c}{$\mathcal{B}(:,:,1)$} & 
\multicolumn{6}{c}{$\mathcal{B}(:,:,2)$} \\
   \hline 
 22&14&25&15&19&14&22&8&38&16&29&12&29&10&20&13\\
 25&22&22&25&19&19&16&22&17&38&20&13&14&29&17&20\\
   \hline
\end{tabular}.\quad
\end{center}
By Lemma~\ref{lemma:5.2}, it is straightforward to verify that the given tensors are compatible. The solution is $\mathcal{X}$
\begin{tabular}{c c c c|c c c c}
\hline
\multicolumn{4}{c}{$\mathcal{X}(:,:,1)$} & 
\multicolumn{4}{c}{$\mathcal{X}(:,:,2)$} \\
   \hline 
 1&3&2&1&3&2&1&3\\
 5&2&3&2&1&2&2&1\\
   \hline
\end{tabular}.
\end{example}
\subsection{The general case corresponds to \texorpdfstring{$m\neq h$}{m=h} with \texorpdfstring{$l=p$}{l=p}}
This section, we examine the existence of solution of the tensor tensor Eq\eqref{eq:5.1}, under the assumption that the tensors $\mathcal{A}$, $\mathcal{B}$, $\mathcal{C}$, and $\mathcal{D}$ are fixed. In parallel with Lemma~\ref{lemma:5.2}, we establish necessary conditions for the existence of solution of Eq\eqref{eq:5.1} and characterize the admissible dimensions of its solutions.
\begin{lemma}\label{lemma:5.3}
Let $X \in \mathbb{C}^{p \times q\times r}$ be a solution of the tensor-tensor Eq~\eqref{eq:5.1}.  
Then the dimensions of the tensors $\mathcal{A}$, $\mathcal{B}$, $\mathcal{C}$, and $\mathcal{D}$ necessarily hold the following conditions:\vspace{-0.5em}
\begin{enumerate}
\item[(i)] The ratios $\frac{h}{m}$ and $\frac{d}{b}$ are positive integers; (ii) The dimensions satisfy
$p = l = \frac{n h}{\alpha m}, \qquad q = \frac{ad}{b} = \frac{k}{\alpha},$
where $\alpha$ is a common divisor of $n$ and $k$. Furthermore,
$(\alpha, \frac{h}{m}) = 1,$
and the solution $X$ is unique.
\end{enumerate}
\end{lemma}
\begin{proof}
The proof follows the analogous arguments as those in Lemma \ref{lemma:4.3} and is therefore omitted.
\end{proof}
\begin{example}
Let $\mathcal{A,\ B,\ C}$ and $\mathcal{D}$ be tensors defined as
     \begin{center}
\begin{tabular}{c c c|c c c}
\hline
  \multicolumn{3}{c}{$\mathcal{A}(:,:,1)$} & 
  \multicolumn{3}{c}{$\mathcal{A}(:,:,2)$} \\
   \hline 
  1&0&-1&2&1&0\\
  2&1&0&0&2&1\\
 \hline
\end{tabular},\quad
   \begin{tabular}{c c c c c c|c c c c c c}
\hline
\multicolumn{6}{c}{$\mathcal{B}(:,:,1)$} & 
\multicolumn{6}{c}{$\mathcal{B}(:,:,2)$} \\
   \hline 
5&0&2&-2&-6&-1&4&-1&1&-1&1&0\\
1&5&0&-1&-2&-6&0&4&-1&6&-1&1\\
2&1&5&0&-1&-2&4&0&4&-2&6&-1\\
2&2&1&4&0&-1&1&4&0&11&-2&6\\
   \hline
\end{tabular},\quad
\end{center}
\begin{center}
\begin{tabular}{c c c|c c c}
\hline
  \multicolumn{3}{c}{$\mathcal{C}(:,:,1)$} & 
  \multicolumn{3}{c}{$\mathcal{C}(:,:,2)$} \\
   \hline 
 3&2&-2&4&0&1\\
 1&3&2&0&4&0\\
   \hline
\end{tabular}\quad
and\quad
\begin{tabular}{c c c|c c c}
\hline
\multicolumn{3}{c}{$\mathcal{D}(:,:,1)$} & 
\multicolumn{3}{c}{$\mathcal{D}(:,:,2)$} \\
   \hline 
  11&-2&0&9&1&-5\\
  10&14&13&2&23&-4\\
   \hline
\end{tabular}. 
\end{center}
By Lemma \ref{lemma:5.3}, it is straightforward to verify that the given tensors are compatible. The solution is $\mathcal{X}$
\begin{tabular}{c c|c c}
\hline
\multicolumn{2}{c}{$\mathcal{X}(:,:,1)$} & 
\multicolumn{2}{c}{$\mathcal{X}(:,:,2)$} \\
   \hline 
 1&0&2&-1\\
 0&6&1&-1\\
   \hline
\end{tabular}.
\end{example}
\subsection{The general case corresponds to \texorpdfstring{$m\neq h$}{m=h} with \texorpdfstring{$l\neq p $}{l=p}}
This section, we examine the existence of solution of the tensor-tensor Eq\eqref{eq:5.1} for the case $m \neq h$ and $p \neq l$, with the tensors $\mathcal{A}$, $\mathcal{B}$, $\mathcal{C}$, and $\mathcal{D}$ know. A necessary condition for the existence of a solution can be obtained directly from the definition of the STP.
\begin{lemma}\label{lemma:5.5}
Assume that the tensor Eq\eqref{eq:5.1} admits a solution $X \in \mathbb{C}^{p\times q\times r}$. Then the dimensions of the tensors $\mathcal{A}$, $\mathcal{B}$, $\mathcal{C}$, and $\mathcal{D}$ necessarily hold the following conditions:\vspace{-0.5em}
\begin{enumerate}
\item[(i)] The ratios $\frac{h}{m}$ and $\frac{d}{b}$ are positive integers; (ii) The dimensions of $X$ are given by
$p=\frac{l}{\beta}=\frac{nh}{m\alpha},\
q=\frac{k}{\alpha}=\frac{ab}{b\beta},$
where $\alpha$ and $\beta \neq 1$ denote the common divisors of $k$ with $n$ and $l$ with $a$, respectively. Furthermore,
$
(\alpha,\frac{h}{m})=1 
\ \text{and} \
(\beta,\frac{d}{b})=1 .
$
\end{enumerate}
\end{lemma}
\begin{proof}
The proof follows the analogous arguments as those in Lemma~\ref{lemma:4.8} and is therefore omitted.
\end{proof}
\begin{remark}
If a solution has dimensions $p \times q \times r$ that satisfy the conditions outlined in Lemma~\ref{lemma:5.5}, these dimensions are referred to as a compatible size. Following the reasoning in Remark~\ref{remark:5.1}, assume that
$\frac{q_2}{p_1} = \frac{p_2}{q_1} > 1,$
and let $X \in \mathbb{C}^{p \times q}$ be a solution to the tensor-tensor Eq\eqref{eq:5.1}. Then, the tensor
$X_1 = X \otimes I_{\frac{q_2}{p_1}\times \frac{q_2}{p_1}}$
also satisfies Eq\eqref{eq:5.1}, with dimensions $X_1 \in \mathbb{C}^{p_2 \times q_2 \times r}$. Moreover, if the solution $X_1$ of compatible size $p_2 \times q_2 \times r$ is unique, it follows that the solution $X$ of compatible size $p_1 \times q_1 \times r$ is also unique.
\end{remark}\vspace{-0.5em}
Analogous to Theorem \ref{theomer:4.8}, we characterize the explicit structure of $\mathcal{B, \ D}$ using the properties of the STP and the dimensional compatibility condition.\vspace{-0.5em}
\begin{theorem}
If the tensor-tensor Eq\eqref{eq:5.1} admits a solution $X \in \mathbb{C}^{p \times q\times r}$, then tensors $\mathcal{B}$ and $\mathcal{D}$ must be block-partitioned, with each block being a Toeplitz tensor. Specifically, $\mathcal{B}$ and $\mathcal{D}$ can be represented in the following block forms:
\begin{equation}
\mathcal{B} =
\begin{bmatrix}
\text{Block}_{11}(\mathcal{B}) & \cdots & \text{Block}_{1q}(\mathcal{B}) \\
\vdots & \ddots & \vdots \\
\text{Block}_{m1}(\mathcal{B}) & \cdots & \text{Block}_{mq}(\mathcal{B})
\end{bmatrix},
\quad
\mathcal{D} =
\begin{bmatrix}
\text{Block}_{11}(\mathcal{D}) & \cdots & \text{Block}_{1b}(\mathcal{D}) \\
\vdots & \ddots & \vdots \\
\text{Block}_{p1}(\mathcal{D}) & \cdots & \text{Block}_{pb}(\mathcal{D})
\end{bmatrix}.
\end{equation}
where $\operatorname{Block}_{sj}(\mathcal{B}) \in \mathbb{C}^{n \times k\times r}$ and
$\operatorname{Block}_{gt}(\mathcal{D}) \in \mathbb{C}^{a \times b\times r}$ are Toeplitz tensors, with
$j = 1,\ldots,q$, $t = 1,\ldots,b,\ g=1,\ldots,p$ and $s=1,\ldots,m$.
\end{theorem}
Next, we provide example illustrating the above Lemma \ref{lemma:5.5}
\begin{example}
Let $\mathcal{A,\ B,\ C}$ and $\mathcal{D}$ be tensors defined as
     \begin{center}
\begin{tabular}{c|c}
\hline
  \multicolumn{1}{c}{$\mathcal{A}(:,:,1)$} & 
  \multicolumn{1}{c}{$\mathcal{A}(:,:,2)$} \\
   \hline 
 3&-1\\
 4&2\\
 \hline
\end{tabular},\quad
   \begin{tabular}{c c|c c}
\hline
\multicolumn{2}{c}{$\mathcal{B}(:,:,1)$} & 
\multicolumn{2}{c}{$\mathcal{B}(:,:,2)$} \\
   \hline 
5&9&1&-3\\
-7&5&15&1\\
10&12&8&6\\
4&10&14&8\\
   \hline
\end{tabular},\quad
\begin{tabular}{c c|c c}
\hline
  \multicolumn{2}{c}{$\mathcal{C}(:,:,1)$} & 
  \multicolumn{2}{c}{$\mathcal{C}(:,:,2)$} \\
   \hline 
 2&-2&1&4\\
 1&2&0&1\\
 3&1&-1&0\\
 0&3&2&-1\\
 \hline
\end{tabular}\quad
and\quad
\begin{tabular}{c c|c c}
\hline
\multicolumn{2}{c}{$\mathcal{D}(:,:,1)$} & 
\multicolumn{2}{c}{$\mathcal{D}(:,:,2)$} \\
   \hline 
  14&3&1&6\\
  2&14&7&1\\
  7&20&8&-11\\
  1&7&8&8\\
   \hline
\end{tabular}. 
\end{center}
By Lemma \ref{lemma:5.5}, it is straightforward to verify that the given tensors are compatible. The solution is $\mathcal{X}$
\begin{tabular}{c c|c c}
\hline
\multicolumn{2}{c}{$\mathcal{X}(:,:,1)$} & 
\multicolumn{2}{c}{$\mathcal{X}(:,:,2)$} \\
   \hline 
 2&3&1&0\\
 -1&2&4&1\\
   \hline
\end{tabular}.
\end{example}
\section{Conclusion}
In this paper, we studied the couple tensor equation $\mathcal{A}\ltimes \mathcal{X}=\mathcal{B},\ \mathcal{X}\ltimes\mathcal{C}=\mathcal{D}$
under the framework of the STP combined with the $t$-product. The necessary and sufficient conditions are derived for the existence of solutions, and an equivalence criterion for solvability is established. These results were obtained for the cases where the unknown $\mathcal{X}$ is a vector, a matrix, or a higher-order tensor. In addition, the explicit structural properties of the tensor $\mathcal{B}$ and $\mathcal{D}$, including Toeplitz and Circulant structures, were analyzed. 

The framework developed in this paper can be naturally extended to the STP combined with the $c$-product. In particular, the compatibility conditions derived for the $t$-product remain necessary for the corresponding tensor equation under the $c$-product. However, the explicit structure of the solutions may differ from that obtained in the $t$-product setting. Further investigation is required to analyze and characterize the structure of $\mathcal{B}$ and $\mathcal{D}$ under the semi-tensor product with the $c$-product.

A tensor-based framework for 3D brain tumor segmentation using the coupled system $\mathcal{A}\ltimes \mathcal{X}=\mathcal{B},\; \mathcal{X}\ltimes \mathcal{C}=\mathcal{D},$
where the latent tensor $\mathcal{X}\in \mathbb{R}^{H\times W\times S}$ preserves full volumetric structure. 
The operators $\mathcal{A}$ and $\mathcal{C}$ enforce spatial regularization and task-specific transformation, respectively, while $\mathcal{B}$ and $\mathcal{D}$ represent intermediate and final segmentation outputs.
This coupled formulation constrains $\mathcal{X}$ from both sides, improving stability and identifiability compared to single-equation approaches. 
The model effectively captures spatio-anatomical features and supports accurate tumor segmentation with enhanced boundary quality and reduced false positives, making it suitable for volumetric MRI analysis and hybrid learning frameworks.
 \begin{figure}[H]
\centering
\includegraphics[width=10cm, height=5cm]{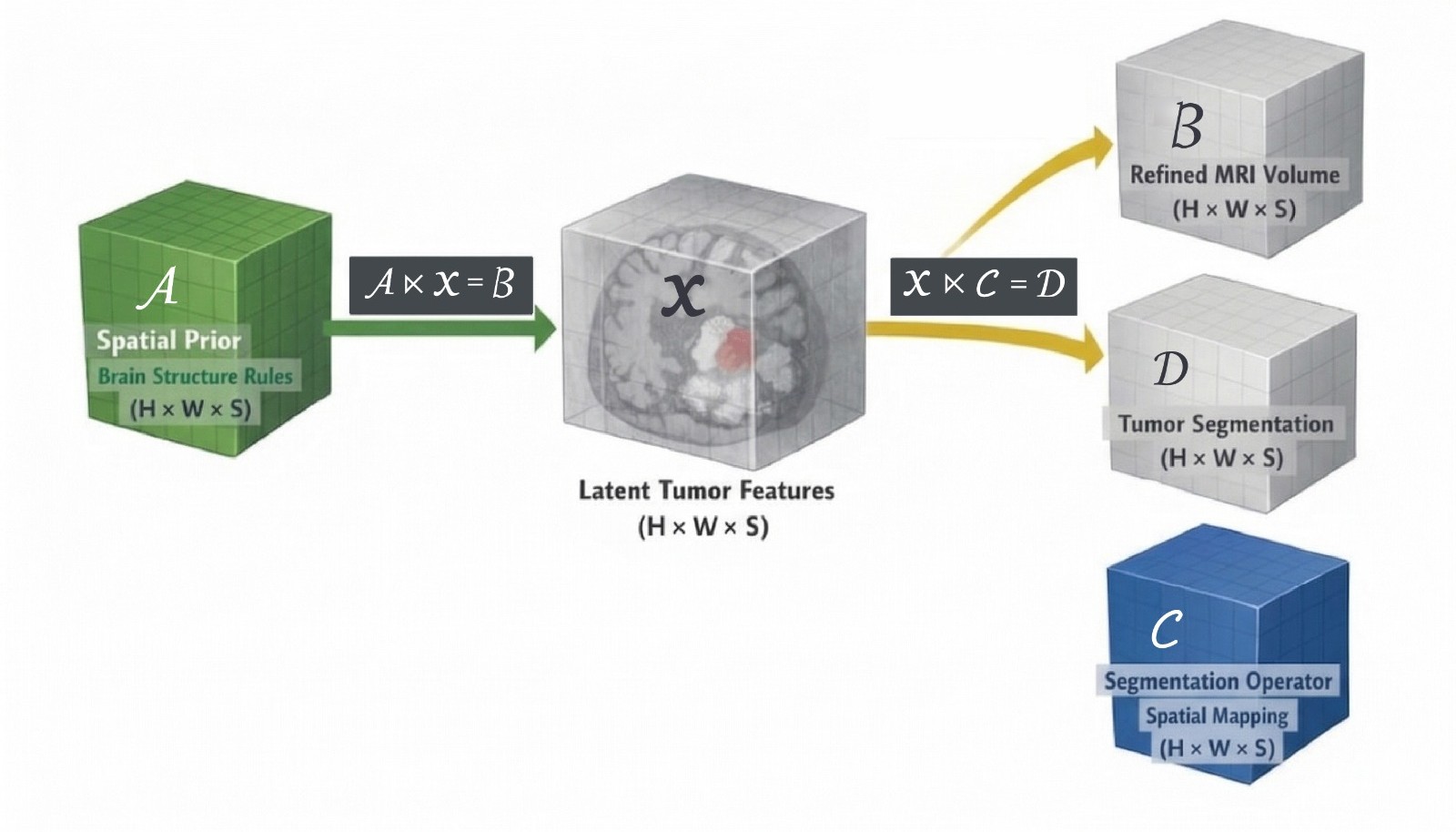}
\caption{3D brain tumor segmentation framework}
\label{fig:example}
\end{figure}
\section*{Conflicts of interest}
The authors declare that they have no conflict of interest
\section*{Data Availability Statements}
 Data sharing is not applicable to this article as no new data is analyzed in this study.

\end{document}